\documentclass[english,letterpaper,11pt,reqno]{amsart}
\usepackage{appendix}
\usepackage{amsbsy}
\usepackage{amsfonts}
\usepackage{amsmath}
\usepackage{amssymb}
\usepackage{amsthm}
\usepackage{graphicx}
\usepackage{ifthen}
\usepackage{textcomp}
\usepackage{enumitem, color, amssymb}
\usepackage{mathrsfs}
\usepackage[bookmarksnumbered,colorlinks]{hyperref}
\hypersetup{colorlinks=true, linkcolor=blue, citecolor=blue}
\usepackage{placeins}

\newtheorem {lemma} {Lemma} [section]
\newtheorem{thm}{Theorem}
\newtheorem{prop}[lemma]{Proposition}

\theoremstyle{remark}

\newcounter{nmdthmcnt}

\newcommand{\beqa}{\begin{eqnarray}}
\newcommand{\beq}{\begin{equation}}
\newcommand{\eeqa}{\end{eqnarray}}
\newcommand{\eeq}{\end{equation}}

\newcommand{\be}{\begin{equation}}
\newcommand{\ee}{\end{equation}}
\newcommand{\lb}[1]{\label{#1}}

\renewcommand{\Ref}[1]{(\ref{#1})}

\newcommand\kk{{\bf k}}

\newcommand\xx{{\bf x}}
\newcommand\yy{{\bf y}}
\newcommand\kf{\hat\kk}
\newcommand\tf{\hat\tT}
\newcommand\xf{\hat\xx}
\newcommand\yf{\hat\yy}
\newcommand\kv{\kk}
\newcommand\tv{\tT}
\newcommand\xv{\xx}
\newcommand\yv{\yy}

\newcommand\n{\nabla}

\newcommand{\HH}{\mathcal{H}}
\newcommand{\VV}{\mathcal{V}}
\newcommand{\UU}{\mathcal{U}}
\newcommand{\ta}{\tau}
\newcommand\om{\omega}
\newcommand\tT{{\bf t}}

\newcommand{\al}{\alpha}
\newcommand{\bet}{\beta}

\newcommand{\we}{\wedge}

\newcommand{\sig}{\sigma}
\newcommand{\del}{\delta}

\newcommand{\lam}{\lambda}
\newcommand{\fr}{\frac}

\begin{document}
\title[]{Cohomogeneity one central K\"ahler metrics in dimension four}
%On the completeness of some Bianchi type A and related central K\"ahler metrics}
\author[]{Thalia Jeffres and Gideon Maschler}
\address{Wichita State University, Wichita, KS}
\email{jeffres@math.wichita.edu}
%\email{Aaazami@clarku.edu\,,\,Gmaschler@clarku.edu}
\address{Department of Mathematics and Computer Science\\ Clark University\\ Worcester, MA }
\email{gmaschler@clarku.edu}
%\email{gmaschler@clarku.edu}

\maketitle
\thispagestyle{empty}
\begin{abstract}
A K\"ahler metric is called central if the determinant of its Ricci endomorphism
is constant \cite{m}. For the case in which this constant is zero, we study on $4$-manifolds
the existence of complete metrics of this type which are cohomogeneity one for three unimodular
$3$-dimensional Lie groups: $SU(2)$, the group of Euclidean plane motions $E(2)$
and a quotient by a discrete subgroup of the Heisenberg group $\mathrm{nil}_3$. We obtain
a complete classification for $SU(2)$, and some existence results for the other
two groups, in terms of specific solutions of an associated ODE system.
\end{abstract}

%-----------------------
\section{Introduction}

In this paper the term central K\"ahler metric refers to a K\"ahler metric
for which the determinant of the Ricci endomorphism is constant. This is a special case of the
metric type called central in \cite{m}. Riemannian and hermitian metrics with constant Ricci determinant were considered earlier, see for example \cite{k, l, b-m}.

On a compact K\"ahler manifold there exists a Futaki-type invariant for central K\"ahler metrics
\cite{f-t}. An associated functional analogous to the K-energy appears in \cite{c-t,s-w,t,r}. If a compact manifold admits a K\"ahler-Einstein metric, it is shown in \cite{m} that a central K\"ahler metric also exists in any K\"ahler class, and an appropriate notion of uniqueness holds for it as well.
If a compact manifold with a definite first Chern class admits a central K\"ahler metric, it also
admits a K\"ahler-Einstein metric. It is, as far as we know, an open question whether in the case
where the first Chern class has no sign, a similar result holds with the conclusion that the manifold
admits a K\"ahler metric with constant Ricci eigenvalues.

On noncompact manifolds the methods for obtaining the above results are unavailable, and existence of complete central K\"ahler metrics does not seem to have been explored. The main purpose of this paper is to demonstrate existence of such metrics which are also invariant under certain cohomogeneity one group actions on $4$-manifolds. For technical reasons our results are limited to central metrics with {\em zero Ricci determinant}, which we call centrally flat, or metrics of zero central curvature. Note
that in the rough classification in \cite{m} of compact complex surfaces admitting central K\"ahler metrics, the most difficult and least understood case is the centrally flat one.

It should be noted that the groups we consider are not always compact.
More specifically, up to a possible quotient by a discrete subgroup, the groups are three of the six
unimodular $3$-dimensional Lie groups. These are (a quotient of) the Heisenberg group
$\mathrm{nil}_3$,  $SU(2)$ and the group of Euclidean plane motions $E(2)$. For the first and second of these, closely related incomplete central metrics appear in \cite[Thm. 1 and sec. 3.5]{a-m2}.

Our methods involve ODE techniques, and are directly inspired by the papers \cite{d-s1,d-s2} of Dancer and Strachan, and the recent articles \cite{a-m2,mr}. In all of these the K\"ahler-Einstein case is
prominent. Another less closely related work is \cite{mr1}, which examines K\"ahler-Ricci solitons
for actions of Heisenberg groups also in higher dimensions. For all the metrics we find,
completeness holds on manifolds admitting a singular orbit, and the smooth extension of the metric
and K\"ahler form to this orbit are shown using the recent systematic approach of Verdiani and Ziller \cite{vz}.

We remark that the need to restrict ourselves to centrally flat metrics is due to the rather unexpected fact that the Center Manifold Theorem applies only in this case to our systems of
ODEs. Throughout the paper we are, of course, only interested in centrally flat metrics which are
not Ricci-flat.

It is interesting to compare our results to those for K\"ahler-Einstein metrics in the above
references. We note first that our results are restricted to metrics which are diagonal in an appropriate coframe containing left invariant $1$-forms for the group. Note that cohomogeneity one K\"ahler-Einstein metrics under $SU(2)$ must be diagonal, but we are not aware of a similar result for central metrics.

For the action of $SU(2)$, we classify the possible cases (Theorem~\ref{thm3}), but our methods yield only complete diagonal centrally flat metrics which are biaxial, meaning that two out of three metric coefficients are equal. In contrast, \cite{d-s1} also find complete triaxial K\"ahler-Einstein metrics (in which the three coefficients are all distinct). Finally, the biaxial centrally flat metrics we find, just like the corresponding K\"ahler-Einstein ones in \cite{d-s1}, can be given in explicit form.
%Furthermore, apart from some incomplete examples, our metrics are not given explicitly, whereas %there are known explicit examples of complete biaxial K\"ahler-Einstein metrics.

For $E(2)$, we obtain inexplicit triaxial metrics in analogy with the same result in \cite{mr} in the K\"ahler-Einstein case (see Theorem~\ref{thm4}). However, in that article all cases are classified, whereas  for centrally flat metrics, we have to exclude one case from consideration, as we only find for it partial information concerning solutions satisfying a certain analyticity property.

The complete centrally flat metrics under the action of the quotient of $\mathrm{nil}_3$ are explicitly given examples. See Theorem~\ref{thm2}.

In sections \ref{sec:sh} and \ref{sec:construct} and the appendix,  we recall the ansatz of \cite{mr}, based on the notion of shear operators, and adopt it to the case of central metrics.
As in the K\"ahler-Einstein case, this ansatz may include more than just cohomogeneity one examples.  Here we include it mainly to connect with that work, and recall its specialization to the cohomogeneity one case in section \ref{sec:coho}. Our main results are given
in sections \ref{sec:heis}, \ref{sec:SU2} and \ref{sec:Euc}.

%----------------------
\section{Shear and integrability}\lb{sec:sh}

Let $(M,g,J)$ be an almost hermitian $4$-manifold. We fix a local oriented orthonormal frame
denoted
\[
\{e_i\}=\{\kk, \tT=J\kk, \xx, \yy=J\xx\}.
\]
In the frame domain, we have an orthogonal decomposition of the tangent bundle:
\[
\text{$TM=\VV\oplus\HH$, \quad with ${\VV}=\mathrm{span}(\kk,\tT),\quad {\HH}=\mathrm{span}(\xx,\yy)$.}
\]
Let $\UU$ stand for either $\VV$ or $\HH$, and $\pi_{\UU^\perp}:TM\to{\UU^\perp}$ denote the orthogonal projection.
For a vector field $X\in\Gamma(\UU)$, consider the operator $\pi_{\UU^\perp}\circ \n X|_{\UU^\perp}:\Gamma(\UU^\perp)\to\Gamma({\UU^\perp})$, where $\n$ is the Levi-Civita covariant
derivative of $g$. Define the {\em shear operator} of $X$ by
\[
\text{$S_X$\ :=\ trace-free symmetric part of $\pi_{\UU\perp}\circ \n X|_{\UU^\perp}$.}
\]
%See \cite{a-m} for background on the relation to the shear operator in general relativity.

Recall the condition for integrability of $J$ in terms of shear operators given in \cite{a-m, mr}.
\begin{thm}\lb{Nij0}
Given the above set-up, the almost complex structure $J$ is integrable in the frame domain
if and only if
\begin{align}\lb{Nij}
\mathrm{i})&\ \ J S_{\xx}=S_{\yy} \text{ on ${\VV}$.}\nonumber\\
\mathrm{ii})&\ \ J S_{\kk}=S_{\tT} \text{ on ${\HH}$.}
%\mathrm{and}\quad \mathrm{ii})\ J\n^o\kk_+=\n^o\!J\kk_+ \text{ on ${\HH}$.}
\end{align}
\end{thm}
In application we will also rely on the following expression of the matrix corresponding to the
shear operator in a local oriented orthonormal frame $\{v_1,v_2\}$ on $\UU^\perp$.
\[
[S_X]_{v_1,v_2}=
\begin{bmatrix}%{cc}
        - \sigma_1 & \sigma_2\\
        \sigma_2 & \sigma_1\\
      \end{bmatrix},
\]
with {\em shear coefficients}:
\be\lb{sh-coef}
\begin{aligned}
2\sigma_1\ &:=\
% g(\n_\yy X,\yy) - g(\n_\xx X,\xx)=
 \ \ g([X,v_1],v_1)-g([X,v_2],v_2),\\
2\sigma_2\ &:=\
% g(\n_\yy X,\xx) + g(\n_\xx X,\yy)=
 -g([X,v_1],v_2) - g([X,v_2],v_1).
\end{aligned}
\end{equation}

One simple case in which integrability holds by Theorem \ref{Nij0}
is when all the shears vanish: $S_{e_i}=0$, $i=1,\ldots,4$.
We refer to this as the shear-free case.
%Our main results concern cases which are not shear-free.

%-----------------------
\section{Shear and K\"ahler metrics}\lb{sec:construct}

We recall here an ansatz for K\"ahler metrics on $4$-manifolds given in \cite{mr}. Let$(M,g, J)$ be an almost hermitian $4$-manifold admitting an orthonormal frame $\{e_i\}=\{\kk, \tT, \xx, \yy\}$,
with $J\kk=\tT$, $J\xx=\yy$, defined over an open $U\subset M$, which satisfies the Lie bracket relations
%\be\lb{brack}
\begin{align}
&[\kk,\tT]=L(\kk+\tT),\qquad &&[\xx,\yy]=N(\kk+\tT),\lb{brack1}\\
&[\kk,\xx]=A\xx+B\yy,\qquad  &&[\kk,\yy]=C\xx+D\yy,\lb{brack2}\\
&[\tT,\xx]=E\xx+F\yy,\qquad  &&[\tT,\yy]=G\xx+H\yy,\lb{brack3}
\end{align}
%\end{equation}
for smooth functions $A, B, C, D, E, F, G, H, L, N$ on $U$ such that
%\be\lb{rels}
\begin{align}
A&-D=F+G,\qquad B+C=H-E,\lb{rels1}\\
N&=A+D=-(E+H).\lb{rels2}
\end{align}
Then $(g,J)$ is K\"ahler (see \cite[Prop. 3.1]{mr}).
Its Levi-Civita connection over  $U$ can be given by setting
\be\lb{conn}
\n_\kk\kk=-L\tT,\qquad  \n_\xx\xx=A\kk+E\tT,\qquad
\n_\xx\kk=-A\xx+E\yy,%\qquad &&\n_\kk\xx=(H-C)\yy.
\end{equation}
and then having all other covariant derivative expressions on frame fields
determined by the requirement that $\n$ be torsion-free and make $J$ parallel.

The Ricci form of the K\"ahler metric $g$ was shown in \cite{mr} to take the form
%in Proposition~\ref{kah} is computed
%as follows. Denote by $w_1=\kk-i\tT$, $w_2=\xx-i\yy$ the corresponding complex-valued frame,
%and compute the four complex valued $1$-forms $\Gamma_i^j, i,j=1,2$ for which
%$\n w_i=\Gamma_i^j\otimes w_j$, where here $\n$ denotes the obvious complexification
%of the Levi-Civita connection of $g$. The formulas are deduced by computing the components
%$\n_{e_\ell} w_i$, where $e_\ell$ stands for one of the frame fields, using the covariant derivative %frame formulas for the Levi-Civita connection $\n$, given in the proof of Proposition~\ref{kah}.
%Two of the four $\Gamma_i^j$'s resulting from this calculation are
%\[
%\Gamma_1^1=-iL(\hat\kk+\hat\tT),\qquad \Gamma_2^2=-i(C-H)\hat\kk-i(A-F)\hat\tT,
%\]
%where the hatted quantities denote  the non-metrically-dual coframe of $\{e_\ell\}$.
%Citing, for example, Lemma 4.2 in \cite{dr-ma}, the Ricci form of $g$ is given by
\begin{multline}\lb{ric}
\rho=
%i(d\Gamma_1^1+d\Gamma_2^2)=
L(d\hat\kk+d\hat\tT)+(C-H)d\hat\kk+(A-F)d\hat\tT\\
+dL\we(\hat\kk+\hat\tT)+d(C-H)\we\hat\kk+d(A-F)\we\hat\tT.
\end{multline}
where the hatted quantities denote  the dual coframe of $\{e_\ell\}$.

Using formulas \Ref{d-frame} in the appendix for the exterior derivatives of the coframe
$1$-forms, as well as
$df=d_\kk f\,\hat\kk+d_\tT f\,\hat\tT+d_\xx f\,\hat\xx+d_\yy f\,\hat\yy$, valid for a smooth
function $f$ on $M$, we can rewrite this formula in the form
\begin{equation*}
\rho=\al\hat\xx\we\hat\yy+\bet\hat\kk\we\hat\tT+\gamma\hat\kk\we\hat\xx+\del\hat\kk\we\hat\yy+
\phi\hat\tT\we\hat\xx+\psi\hat\tT\we\hat\yy,
\end{equation*}
where
\begin{align}
\al&=-N(2L+C-H+A-F),\nonumber\\
\bet&=-L(2L+C-H+A-F)+d_{\kk-\tT}L-d_\tT(C-H)+d_\kk(A-F),\nonumber\\
\gamma&=-d_\xx(L+C-H),\nonumber\\
\del&=-d_\yy(L+C-H),\nonumber\\
\phi&=-d_\xx(L+A-F),\nonumber\\
\psi&=-d_\yy(L+A-F).\lb{ric-coeff}
\end{align}

%----------------------------------------
%\section{The central metric condition}
The central curvature $c$ is defined by the equation
\[
\rho^{\we 2}=c\,\om^{\we 2} . \]
If $c$ is constant, we write
\( c=\lam,
\) and call the corresponding metric a {\em central metric}.
In terms of the Ricci coefficients \Ref{ric-coeff}, we then have
\be\lb{cent}
c=\al\bet-\gamma\psi+\del\phi=\lam,
\end{equation}
because $\rho^{\we 2}=2c\,\hat\xx\we\hat\yy\we\hat\kk\we\hat\tT$ while
$\om^{\we 2}=2\hat\xx\we\hat\yy\we\hat\kk\we\hat\tT$. Such a central
metric will not be Einstein if either at least one of $\gamma$, $\delta$, $\phi$, $\psi$
is not identically zero or $\al$ and $\bet$ are not both equal to the same constant.

%\section{The ODE and generalized PDE systems}\lb{OdE}

We now recall a function built in to our ansatz that gave rise in \cite{mr} to the independent
variable in a system of ODEs used in both \cite{mr} and \cite{mr1}.

The Lie bracket relations \Ref{brack1}-\Ref{brack3} imply that the distribution spanned by
$\kk+\tT$, $\xx$ and $\yy$ is integrable. Since this distribution is orthogonal to $\kk-\tT$,
while the latter vector field has constant length and is easily seen to have geodesic flow, it
follows that it is locally a gradient (cf. \cite[Cor. 12.33]{onel}). Thus, there exists
a smooth function $\ta$ defined in some open set $V\subset U$, such that
\be\lb{grad}
\kk-\tT=\n\ta.
\end{equation}
Consider now the six functions $P$, $Q$, $R$, $S$, $L$, $N$,
where the last two are as in \Ref{brack1}, and the first four
are given in terms of four of the functions in \Ref{brack2}-\Ref{brack3} by
\begin{align}
P&=(B-C)+(F-G), &&Q=(B-C)-(F-G),\nonumber \\
R&=\sqrt{(B+C)^2+(F+G)^2}, &&S=\tan^{-1}\left(\frac{B+C}{F+G}\right),\lb{chan-var}
%-\tan^{-1}k,  \text{ for a constant $k$.}
\end{align}
where $S$ is only defined on the set $\{F+G\}\ne 0$.

In terms of these variables, it is shown in the appendix that in case $A, B, \ldots H$, $L$ and $N$
are each a composition of a function of $\ta$, the ansatz equations simplify to five ODEs \Ref{cent-onevar}
involving those functions. In particular the ODE giving the central curvature equation
takes the form
\be
-N(2L+N-P/2)[-L(2L+N-P/2)+(2L'+N'-P'/2)]=\lam.\lb{cent-onevar0}
\end{equation}

\section{Cohomogeneity one examples}\lb{sec:coho}
It was shown in \cite{mr} that the ansatz of section \ref{sec:construct} includes as a special case
cohomogeneity one diagonal K\"ahler metrics under the action of a unimodular group in dimension three. In this section we review their construction, and derive the central metric equation for such metrics.

Assume that $(M,g)$ is a $4$-dimensional Riemannian manifold admitting a proper isometric action by a three dimensional Lie group $\mathcal{G}$ with cohomogeneity one having a discrete isotropy group.
Assuming also that $\mathcal{G}$ is a unimodular group, we choose a frame of left-invariant vector fields
$X_1$, $X_2$, $X_3$, and dual coframe consisting of left-invariant $1$-forms $\sig_1$, $\sig_2$, $\sig_3$. These satisfy
\begin{align}
%[\partial_t,X_i] &= 0,\quad i=1,\ldots,3,\nonumber  \\
[X_1,X_2] &= -p_3 X_3,   &&d\sigma_1 = p_1\sigma_2\wedge\sigma_3,\nonumber \\
[X_2,X_3] &= -p_1 X_1,   &&d\sigma_2 = p_2\sigma_3\wedge\sigma_1,\nonumber\\
[X_3,X_1] &= -p_2 X_2.   &&d\sigma_3 = p_3\sigma_1\wedge\sigma_2.\lb{X123}
\end{align}
for some constants $p_1$, $p_2$, $p_3$.
Cohomogeneity one metrics for such groups are also described as having Bianchi type A.
A diagonal such metrics takes the form
\begin{equation}\label{bianchiAmet} g = (abc)^2dt^2+a^2\sigma_1^2+b^2\sigma_2^2+c^2\sigma_3^2, \end{equation}
for functions $a$, $b$, $c$ of $t$. We note that, of course considering the orthogonal frame
$\partial_t,X_1,X_2,X_3$ dual to $dt,\sigma_1,\sigma_2,\sigma_3$, on $M$,
$\partial_t$ commutes, of course, with all $X_i$,
$i=1,2,3$.

Following Dancer and Strachan \cite{d-s1}, denoting $w_1=bc$, $w_2=ac$, and $w_3=ab$, we
define functions $\alpha$, $\beta$, and $\gamma$ so that
\begin{align}
w_1'&=p_1w_2w_3+\alpha w_1, \\
w_2'&=p_2w_1w_3+\beta w_2, \\
w_3'&=p_3w_1w_2+\gamma w_3.
\end{align}
They show that (modulo reordering the frame vectors) the only K\"ahler structures $(M,g,J)$ with $g$ of the form \Ref{bianchiAmet} have complex structure determined by
\begin{equation}\label{complexJ} J\partial_t = abX_3 \quad\text{and}\quad JX_1=\frac{a}{b}X_2, \end{equation}
and $\alpha$, $\beta$, and $\gamma$ satisfy
\[ \alpha=\beta \quad \text{and}\quad \gamma=0.\]
The K\"ahler form is then given by
\begin{equation}\label{kahlerform} \omega = abc^2dt\wedge\sigma_3+ab\sigma_1\wedge\sigma_2 = w_1w_2dt\wedge\sigma_3+w_3\sigma_1\wedge\sigma_2, \end{equation}
and $w_1,w_2,w_3$ satisfy
\begin{align}
w_1'&=p_1w_2w_3+\alpha w_1,\nonumber \\
w_2'&=p_2w_1w_3+\alpha w_2,\nonumber \\
w_3'&=p_3w_1w_2.\lb{w33}
\end{align}
This, in terms of $a,b,c$ implies
\begin{align}
2a'/a&=-p_1a^2+p_2b^2+p_3c^2,\lb{K1} \\
2b'/b&=p_1a^2-p_2b^2+p_3c^2, \lb{K2}\\
2c'/c&=p_1a^2+p_2b^2-p_3c^2+2\alpha.\lb{KKE}
\end{align}

%In the next subsection we derive the central metric
%Einstein
%condition, after recalling how
%this model fits within the framework of section~\ref{construct}.

%\subsection{The frame $\{\mathbf{k,t,x,y}\}$}

We recall the prescription that makes this model fit with the ansatz of Section $3$.
The orthonormal frame and dual coframe are given  by
\begin{align*}
\mathbf{k} &= \frac{\sqrt{2}}{2}\left(\frac{1}{c}X_3+\frac{1}{abc}\partial_t \right), & \hat{\mathbf{k}} &= \frac{\sqrt{2}}{2}(c\sigma_3+abcdt), \\
\mathbf{t} &= \frac{\sqrt{2}}{2}\left(\frac{1}{c}X_3-\frac{1}{abc}\partial_t \right), & \hat{\mathbf{t}} &= \frac{\sqrt{2}}{2}(c\sigma_3-abcdt), \\
\mathbf{x} &= \frac{X_1}{a}, & \hat{\mathbf{x}} &= a\sigma_1, \\
\mathbf{y} &= \frac{X_2}{b}, & \hat{\mathbf{y}} &= b\sigma_2.
\end{align*}
One can easily check that relations \Ref{brack1}-\Ref{brack3} hold with these choices.

%It can easily be checked that this frame satisfies \Ref{brack1}-\Ref{brack3}
%for the functions
%Then the variables $A,B,C,D,E,F,G,H,L,N$ are
Next the functions of the ansatz are given in terms of $a$, $b$, $c$, by
\begin{align*}
A&=-E=-\frac{a'}{\sqrt{2}a^2bc}=-\frac{1}{a}\frac{da}{d\tau}, & B &=F= -\frac{bp_2}{\sqrt{2}ac}, \\
D&=-H=-\frac{b'}{\sqrt{2}ab^2c}=-\frac{1}{b}\frac{db}{d\tau}, & C&=G=\frac{ap_1}{\sqrt{2}bc}, \\
L&=-\frac{c'}{\sqrt{2}abc^2}=-\frac{1}{c}\frac{dc}{d\tau}, & N &= -\frac{cp_3}{\sqrt{2}ab}.
\end{align*}
Here the prime denotes differentiation with respect to $t$, while the expressions in terms of
$d/d\ta$ hold due to the relation between $\ta$ and $t$ given by
\[ \hat{\mathbf{k}}-\hat{\mathbf{t}}=d\tau=\sqrt{2}abcdt, \qquad\frac{d}{d\tau}=\frac{1}{\sqrt{2}abc}\frac{d}{dt}. \]

Finally, we give the functions $P,Q,R,S$ of the change of
variables \Ref{chan-var}.
\begin{align*}
P&=-\sqrt{2}\frac{a^2p_1+b^2p_2}{abc}, & Q &= 0, \\
R&=\frac{a^2p_1-b^2p_2}{abc}, & S &= \frac{\pi}{4}.
\end{align*}

From the point of view of the ansatz,  the four relations in \Ref{rels1}-\Ref{rels2} that imply the K\"ahler condition
impose only two additional relations here, say $A+D=N$ and $B+C=H-E$,
giving
\begin{align}
\fr{a'}a+\fr{b'}b&=p_3c^2,\lb{K3}\\
\fr{b'}b-\fr{a'}a&=p_1a^2-p_2b^2\lb{another}
\end{align}
which are equivalent to \Ref{K1}-\Ref{K2}. Our remaining task is to determine how the
condition that the metric is central constrain $\al$ in \Ref{KKE}.

The central metric equation \Ref{cent}, is given in the
variables \Ref{chan-var} by \Ref{cent-onevar0}:
\[
-N(2L+N-P/2)[-L(2L+N-P/2)+\fr d{d\ta}(2L+N-P/2)]=\lam.
\]
Calculating using the above formulas for $L$, $N$, $P$  and also \Ref{KKE}, we have
\begin{align*}
2L+N-P/2&=\fr 1{\sqrt{2}abc}\Big(-2\fr{c'}c-p_3c^2+p_1a^2+p_2b^2\Big)\\
&=\fr 1{\sqrt{2}abc}(-2\al).
\end{align*}
So that \Ref{cent-onevar0} takes the form
\begin{align*}
\fr{p_3c}{\sqrt{2}ab}\Big(\fr{-2\al}{\sqrt{2}abc}\Big)
\Big[\fr{c'}{\sqrt{2}abc^2}\fr{-2\al}{\sqrt{2}abc}+\fr 1{\sqrt{2}abc}\Big(\fr{-2\al}{\sqrt{2}abc}\Big)' \,\Big]=\lam.
\end{align*}
A relatively straightforward simplification of this which also uses \Ref{K3}
yields, the equivalent form
\be\lb{alpha}
%\fr{p_3}{(ab)^4c^2}\Big(-p_3c^2\al^2+\fr{(\al^2)'}2\Big)=\lam.
p_3(\al^2)'=2c^2\Big(\lam(ab)^4+p_3^2\al^2\Big)
\end{equation}
This equation, together with \Ref{K1}-\Ref{KKE} constitutes the ODE system for diagonal
Bianchi IX central metrics.

Note that setting $p_3=0$ (hence also $N=0$) forces $\lam=0$ but no other constraints.
On the other hand, setting $\lam=0$ yields, for $p_3\ne 0$ the equation
\[
\al'=p_3c^2\al
\]
which will play a major role in the following sections.

Additionally, one can check that the formula $\al=\mp(\sqrt{\lam}/p_3)(ab)^2$, $\lam>0$ reduces
 \Ref{alpha} to an identity, and this corresponds to the fact that this is the K\"ahler-Einstein
condition (for $p_3\ne 0$, see \cite{mr}). On the other hand, a K\"ahler-Einstein metric with
$p_3=0$ must be Ricci flat ($\lam=0$) and necessarily $\al'=0$. But note in general from \Ref{ric-coeff} that $\al=0$ is necessarily a Ricci flat case, so such solutions will not concern us.

\section{cohomogeneity one central flat metric under a Heisenberg group quotient action.}\lb{sec:heis}

\subsection{The equations}
On the Heisenberg group, with $p_1=0$, $p_2=0$ and $p_3=1$,
equations \Ref{K1}-\Ref{KKE} and \Ref{alpha} take the form
\begin{align}
2\frac{a'}a&=c^2,\lb{K1H}\\
2\frac{b'}b&=c^2,\nonumber\\
2\frac{c'}c&=-c^2+2\al,\lb{KKEH}\\
\al'&=c^2(\al+\lam(ab)^4/\al).\lb{KKEH1}
\end{align}
Since $(a/b)'=0$ is a first integral, $b$ is a constant multiple of $a$, so that potential
metrics are so-called biaxial. From now on we assume this constant is equal to $1$.
Additionally, we adopt the form used in \cite{a-m2,mr1} by making the change of variables
$a^2\,dt=dq$. Then, setting $\phi(q):=a^2$, we see from \Ref{K1H} that
\[
\phi'(q)=2a\fr {da}{dq}=2a\fr {da}{dt}\fr{dt}{dq}=2a\fr {da}{dt}\fr{1}{a^2}=c^2.
\]
It follows that the metric takes the form
\be\lb{g-phi}
g=\phi(q)(\sig_1^2+\sig_2^2)+\phi'(q)(\sig_3^2+dq^2)
\end{equation}
with K\"ahler form
\[
\om=d(\phi(q)\sig_3).
\]
We now use a prime exclusively for the derivative with respect to $q$, while $\al$ will be considered,
depending on the context, as a function of $t$ or a function of $q$.
The two equations \Ref{KKEH}-\Ref{KKEH1} then translate as follows
\begin{align*}
\fr{\phi''(q)}{\phi'(q)}&=2c\fr{dc}{dt}\fr{dt}{dq}\fr 1{c^2}=\left(-c^2+2\al\right)\fr 1{a^2}=
-\fr{\phi'(q)}{\phi(q)}+2\fr{\al}{\phi(q)},\\
\al'(q)&=\fr{d\al}{dt}\fr{dt}{dq}=c^2\Big(\al+\lam\fr{a^8}{\al}\Big)\fr 1{a^2}
=\fr{\phi'(q)}{\phi(q)}\Big(\al+\lam\fr{\phi^4}{\al}\Big).
\end{align*}
Or, simplified
\begin{align}
\fr{(\phi^2)''}{(\phi^2)'}&=2\fr{\al}{\phi},\lb{Heis1}\\
\al'&=\fr{\phi'}{\phi}\Big(\al+\lam\fr{\phi^4}{\al}\Big).\lb{Heis2}
\end{align}

We now set $\lam=0$. Then \Ref{Heis2} implies (if $\al$ is nonzero) that $\al/\phi$ is constant,
which again we choose to be $1$. Substituting this into \Ref{Heis1} gives an
equation with explicit solution
\[
\text{$\phi=C\sqrt{e^{2q}+B}$, \qquad $C>0$, $B$ constants.}
\]
For simplicity we choose $C=1$ and $B=c_1^2$, $c_1>0$.
Then for $q$ real valued, $\phi$ takes values in $(c_1,\infty)$, and
\[
\phi'=\fr{e^{2q}}{(e^{2q}+c_1^2)^{1/2}}.
\]
We show that $g$ is complete, in the next few subsections. Here we point out
that $g$ is not Ricci flat. In fact as $p_3$ is nonzero, the formula near
the end of section \ref{sec:coho} shows that for Ricci flatness we must have $\al=0$,
but in this solution $\al=\phi\ne 0$.

\subsection{Setup}

As in  \cite{a-m2}, in order to avail ourselves of the methods of \cite{vz},
we consider a cohomogeneity one action under the quotient
$\mathcal{G}=\mathrm{nil_3}/\tilde{\mathbb{Z}}$ of the Heisenberg group
by the infinite cyclic group lying in its center, and given by
\[ \tilde{\mathbb{Z}}:= \left\{\begin{bmatrix} 1 & 0 & 2\pi n\\ 0 & 1 & 0\\ 0 & 0 & 1  \end{bmatrix} \Big|\ n\in\mathbb{Z} \right\}.\]
%The value of the real number $r$ will be determined later.
$\mathcal{G}$ has center $K$ isomorphic to
$SO(2)$, whose transitive action on the circle $S^1$ extends to a linear action on $V:=\mathbb{R}^2$. We consider the homogeneous vector bundle $M=\mathcal{G}\times_K V$ (in which points of the product
are identified according to $(g,v)\sim(gk^{-1},kv)$ for $k\in K$).
$\mathcal{G}$ acts on $M$ by left multiplication on the first factor.
%, while $K$ acts by sending $(g,p)\to (gk^{-1}, kp)$.
The action of $\mathcal{G}$ has trivial isotropy at points of a regular orbit, but isotropy $K$ at a
point of the singular orbit $\mathcal{G}/K\approx \mathbb{R}^2$.

\subsection{Length of an escaping curve}

We now choose a left-invariant frame for $\mathcal{G}$ which is given
in coordinates $x$, $y$, $z$ by $X_1=\partial_x$, $X_2=\partial_y+x\partial_z$,
$X_3=-\partial_z$, to which we will add on $M$ the vector field $\partial_q$.
Note that the domain of this coordinate system is open and dense in $M$,
and $z$ is bounded due to the fact that we are considering a quotient.

The corresponding coframe consists of $dq$ and the left invariant coframe for the group
given by $\sig_3=xdy-dz$, $\sig_1=dx$, $\sig_2=dy$.
Given a curve $\gamma(s):I\to M$ of finite length $L(\gamma)$, with coordinate presentation
$(x(s),y(s),z(s),q(s))$, we have
\begin{align*}
\gamma'&=x'\partial_x+y'\partial_y+z'\partial_z+q'\partial_q\\
&=\sig_1(\gamma')X_1+\sig_2(\gamma')X_2+\sig_3(\gamma')X_3+dq(\gamma')\partial_q
\end{align*}
so that
\begin{align*}
g\Big(\gamma',\fr{X_1}{|X_1|}\Big)=x'\sqrt{\phi},\qquad
g\Big(\gamma',\fr{X_2}{|X_2|}\Big)=y'\sqrt{\phi},\qquad
%g\Big(\gamma',\fr{X_3}{|X_3|}\Big)&=(xy'-z')\sqrt{\phi'},
g\Big(\gamma',\fr{\partial_q}{|\partial_q|}\Big)=q'\sqrt{\phi'}.
\end{align*}
It follows that the length of $\gamma$ satisfies the Cauchy-Schwarz estimates
\begin{align}
L(\gamma)=\int_I|\gamma'(s)|\,ds&\ge\inf_I(\sqrt{\phi(q)})\Big|\!\int_I x'\,ds\Big|,\lb{x-prime}\\
L(\gamma)&\ge\inf_I(\sqrt{\phi(q)}) \Big|\!\int_I y'\,ds\Big|,\lb{y-prime}\\
%L(\gamma)&\ge\Big|\int_I\sqrt{\phi'(q)}(xy'-z')\,ds\Big|,\\
L(\gamma)&\ge\Big|\int_I\sqrt{\phi'(q)}q'\,ds\Big|.\lb{q-prime}
\end{align}
Now the right hand side of \Ref{q-prime} equals
\[
\Big|\int_{q(I)}\sqrt{\phi'(q)}\,dq\Big|.
\]
If $q(I)$ has $q=\infty$ as an endpoint, this integral is infinite, so it follows
that this cannot occur if $\gamma$ has finite length. On the other hand
the infima in \Ref{x-prime}-\Ref{y-prime} are positive since this holds for any $q\in\mathbb{R}$.
It then follows from these two equations that for a finite length curve, $x$ and $y$
are bounded. Thus such a curve can only leave every compact set in $M$ if a sequence of
its $q$ values approach $-\infty$. To address this problem we have to attach a ``bolt"
to $M$ at $q=-\infty$,  that is, a singular orbit for the group action, and see that the metric and
K\"ahler form extend smoothly to it.

\subsection{Attaching a bolt}

%\begin{prop}\lb{bolt1}
%For $(M,g)$ as in Theorem~\ref{sol-Heis},
%$g$ and its K\"ahler form can be smoothly extended to a singular fiber over
%$q=q_a$.
%\end{prop}
%\begin{proof}
As $\phi'(q)(dq^2+\sig_3^2)=\frac{d\phi^2}{\phi'(q)}+\phi'(q)\sig_3^2$,
and $\phi'=(\phi^2-c_1^2)/\phi$, the metric
$g$ can be written in the form
\begin{align*}
g&=\fr{\phi}{\phi^2-c_1^2}d\phi^2
+\fr{\phi^2-c_1^2}{\phi}\sig_3^2
+\phi(\sig_1^2+\sig_2^2)\\
\end{align*}
defined on the domain $\phi\in (c_1,\infty)$.

To apply the Verdiani-Ziller smoothness conditions \cite{vz} for a metric at a singular orbit.
We write the metric near $c_1$ in the for $dr^2+h_r$ where $r=0$ corresponds to $\phi=c_1$.
Note that converting equations \Ref{K1H}-\Ref{KKEH1} into this form
amounts to dividing their right hand side by $abc$. From this one can see that a solution can be
extended smoothly if $a$, $b$, $\al$ are even in $r$ and $c$ is odd in $r$. this mean that $\phi=a^2$ and $\phi'=c^2$ are even as functions of $r$.

Computing asymptotically near $c_1$, we have
$dr=\sqrt{\frac{\phi}{\phi^2-c_1^2}}d\phi\approx \sqrt{\frac{c_1}{(\phi-c_1)2c_1}}d\phi$ so
that $r\approx\sqrt{2(\phi-c_1)}$.
Thus near $c_1$
\[
g\approx dr^2+\frac{r^4/4+r^2c_1}{r^2/2+c_1}\sig_3^2+(r^2/2+c_1)(\sig_1^2+\sig_2^2).
\]

We compare this with the smoothness conditions in \cite{vz}, which in our case, for $\mathfrak{m}=\mathrm{span}(X_1,X_2)$, $\mathfrak{p}=\mathrm{span}(X_3)$,
are, near $r=0$,
\begin{align*}
&\text{$g(\mathfrak{m},\mathfrak{m})$ is even in $r$,}\\
&\text{$g(\mathfrak{p},\mathfrak{m})=r^2\psi(r^2)$,}\\
&\text{$g(X,X)=\bar a^2r^2+r^4\xi(r^2)$ for $X\in\mathfrak{p}$.} %such that $g_{\mathrm{euc}}(X,X)|_{r=0}=1$.}
\end{align*}
Only the last condition is not automatic in our case, and in it, $\bar a$ denotes the cardinality of the intersection of the (trivial) stabilizer with $\{\exp(\theta X)\,|\,0\le\theta\le2\pi\}$, with $X$ normalized so that the latter set is a closed one-parameter subgroup.
For the group $\mathcal{G}$ and $X=-X_3$ we have $\bar a=1$ and it is thus sufficient to check
the form of $g(X,X)$ for this $X$. The coefficient of $\sig_3^2$ is
\[ \frac{r^4/4+r^2c_1}{r^2/2+c_1}=r^2-\fr1{4c_1}r^4+O(r^6). \]
Thus the conditions for smoothness of the metric are verified.

The K\"ahler form similarly extends smoothly to the singular fiber. In fact, it is
\begin{multline*} d\phi\we\sig_3+ \phi(\sig_1\we\sig_2)=
d\left[\left(\sqrt{\phi-c_1}\right)^2\right]\we\sig_3+ \phi(\sig_1\we\sig_2)\\
\approx 2^{-1}d(r^2)\we\sig_3+ (r^2/2+c_1)(\hat\xx\we\hat\yy),
\end{multline*}
whereas modifying the conditions in \cite{vz} so that they apply to a $2$-form, shows that
in our case smoothness requires that near $r=0$ the coefficient of $dr\we\sig_3$ has
the form $r\psi(r^2)$ and the coefficient of $\sig_1\we\sig_2$ is even. The fact that $\phi$
is even and the above form conclude the proof. We thus showed
\begin{thm}\lb{thm2}
For every $c_1>0$ the metric
\[
g=\fr{e^{2q}}{\sqrt{e^{2q}+c_1^2}}(dq^2+\sig_3^2)+\sqrt{e^{2q}+c_1^2}(\sig_1^2+\sig_2^2),\qquad q\in\mathbb{R},
\]
defined on the $\mathcal{G}\times_{SO(2)}\mathbb{R}^2$, with $\mathcal{G}\simeq\mathrm{nil}_3/\tilde{\mathbb{Z}}$,
is complete and centrally flat.
\end{thm}
We note that it is not too difficult to classify all complete diagonal cohomogeneity one
centrally flat metrics under the action of $\mathcal{G}$. We demonstrate how to carry this
out in a more difficult case in the next section.

%\end{proof}

\section{Centrally flat metric under the action of the  compact group $SU(2).$ }\lb{sec:SU2}

\subsection{The equations}
We now consider the case where the action is by the compact group \( SU(2) \) for a metric with central curvature $\lambda=0$. With the choices \( p_{1} = p_{2} = p_{3} =1, \) the system \Ref{K1}-\Ref{KKE},\Ref{alpha}
becomes

\begin{eqnarray*}
 a' & = & \frac{a}{2} (-a^{2} +b^{2} +c^{2} ) \\
 b' & = & \frac{b}{2} (a^{2} -b^{2} +c^{2} ) \\
 c' & = & \frac{c}{2} (a^{2} + b^{2} -c^{2} +2\alpha ) \\
 \alpha ' & = & c^{2} \alpha .
\end{eqnarray*}
Note that we must take $\lam=0$ to ensure that the right hand side of the system is smooth,
which allows us to employ the center manifold theorem.

We observe immediately that this can be reduced to a system of three equations in three unknown
functions.
From the first and second equations, it follows that
\[ \frac{d}{dt} (ab) = abc^{2} .\]
Combining this with the fourth equation, we have
\[ \frac{d}{dt} \log \frac{ab}{\alpha } = 0 ,\]
and so there exists a constant $A$ such that
\[ \log \frac{ab}{\alpha } = A, \]
and so \( \alpha = e^{-A} (ab). \) With this, the fourth equation can be eliminated, and the
third equation rewritten. The system becomes

\begin{align}
 a' & =  \frac{a}{2} (-a^{2} + b^{2} + c^{2} ) \nonumber \\
 b' & =  \frac{b}{2} (a^{2} -b^{2} + c^{2} ) \nonumber \\
 c' & =  \frac{c}{2} (a^{2} + b^{2} + 2e^{-A} ab -c^{2} ).\lb{SU2-ode}
\end{align}

%This system is nearly identical to that considered by [DS1], where they investigated the
%K\"{a}hler-Einstein case. Although we will closely follow their method, the details and
%calculations are included here in order that this  paper be readably self-contained.

By a uniqueness argument, if an analytic solution defined on a maximal interval has an
initial value in the region
\[ \mathcal{R} = \{ (a,b,c ) \in \mathbb{R} ^{3} \mid a, b, c >0 \} , \]
then the trajectory will remain in \( \mathcal{R} \) for all values of $t$ for which the solution
exists. From now on we only consider such solutions, for which the metric will defined for
values of $t$ in this interval.

\subsection{Linearization at equilibrium solutions and preliminary calculations}

Equilibrium solutions that lie in \( \overline{\mathcal{R} }  \) are \( (q, q,0), \, (0,q, q), \)
and \( (q, 0,q), \) for \( q\geq 0. \) The coefficient matrix of the linearized system at
\( (q, q,0) \) has eigenvalues \( 0, -2q^{2} , q^{2} (1+e^{-A} ). \) At \( (0,q,q), \) the
eigenvalues are \( q^{2} , 0, -2q^{2}, \) and at \( (q,0,q), \) the eigenvalues are \( 0, q^{2} -2q^{2} . \) In all cases, if \( q>0 \) one of these eigenvalues is positive. The Center Manifold
Theorem guarantees that near an equilibrium solution for which the linearized system
has a positive eigenvalue with no multiplicity, the system admits an unstable
curve.

Next, one verifies
\begin{lemma} For any solution to \Ref{SU2-ode}, we have
\begin{eqnarray*}
 \frac{d}{dt} (ab) & = & (ab) c^{2} \\
 \frac{d}{dt} (ac) & = & ac (b^{2} + e^{-A} ab) \\
 \frac{d}{dt} (bc) & = & bc (a^{2} + e^{-A} ab) \\
 \frac{d}{dt} \Big(\frac{a}{b} \Big) & = & \frac{a}{b} (-a^{2} +b^{2} ) \\
 \frac{d}{dt} \Big(\frac{a}{c} \Big) & = & \frac{a}{c} (-a^{2} - e^{-A} ab +c^{2} )   \\
 \frac{d}{dt} (a^{2} -b^{2} ) & = & (a^{2} -b^{2} ) (-a^{2} -b^{2} +c^{2} ) \\
 \frac{d}{dt} (a^{2} -c^{2} ) & = & -(a^{2} -c^{2} ) (a^{2} -b^{2} +c^{2} ) - 2e^{-A} abc^{2} .
\end{eqnarray*}

\end{lemma}

These are straightforward consequences of the equations of the system.

An immediate implication of these calculations is that in $\mathcal{R}$, the products \( ab, \, ac, \) and \( bc\)
are increasing functions. It follows that at any value of $t,$ at most one of the three can be
decreasing, and also that all three products have finite, non-negative limits as $t$ approaches
the lower endpoint of the maximal interval on which a solution exists. As a further implication, from
the equation for \( d/dt (a/b), \) we see by uniqueness that either $a$ is identically equal to
$b$ or never equal to $b.$ Since the roles of $a$ and $b$ are interchangeable, we may therefore
assume that if $a$ and $b$ do not coincide, that it is $a$ that is greater. Finally, this makes
$b$ a strictly increasing function.

Now choose an initial value in the region \( \mathcal{R} .\) Local existence theory provides for
existence to the initial value problem on a non-empty interval; let \( (\xi , \eta ) \) be the
maximal interval of existence for a given initial value. We investigate whether there are
trajectories which correspond to complete metrics.
%If so, what characterizes or distinguishes these?

\subsection{A maximal solution interval bounded from below}

We will first discover that any candidates for complete metrics correspond to trajectories for
which \( \xi = -\infty .\) This explains the attention paid earlier to the unstable curves, for they are
such trajectories. We then investigate whether any of these do in fact give rise to complete
metrics.

\begin{prop} Trajectories for which \( \xi > -\infty \) correspond to incomplete metrics.

\end{prop}

\begin{proof} The limit of $b$ as \( t\rightarrow \xi \) is zero: If the
limit of \( b \) were non-zero, then both $a$ and $c$ would also have limits, and then all three
functions could be extended continuously to \( \xi \) itself, violating the maximality of
the interval \( (\xi , \eta ). \) Therefore,
\[ \lim _{t\rightarrow \xi } b(t) =0. \]

We observed above that at any particular point, at most one of the functions \( a, b, \) or $c$
can be decreasing. However, under the assumption that \( a\geq b,\) the derivative of $b$ is
positive, so $b$ increases throughout the entire interval of existence. Moreover, it is not
possible that all three increase on all of \( (\xi , \eta ), \) because if they did, then all
three could be extended continuously to the lower endpoint $\xi ,$ contradicting the maximality
of the interval \( (\xi ,\eta ). \) There remain  therefore two possibilities to consider.

{\bf (i)} Suppose at some point \( u\in (\xi ,\eta ), \) that \( a'(u) <0. \) Then
\( -a^{2} +b^{2} +c^{2} <0 \) at this point. Calculating the derivative of this quantity, we find
that
\[ \frac{d}{dt} (-a^{2} +b^{2} +c^{2} ) = a^{4} -(b^{2} -c^{2} )^{2} +2e^{-A} abc^{2} > 0. \]
In other words, whenever \( -a^{2} +b^{2} +c^{2} \) is  negative, the derivative of this quantity
is positive. This implies that if $a$ decreases at any point $u,$ then it decreases on all of
\( (\xi , u ). \) A calculation also shows that where \( a'(t) <0, \) that \( a''(t) >0, \)
and this will be used later.

Since at most one of the three functions can decrease at a point or on an interval, $c$ must be
increasing on all of \( (\xi ,u). \) Therefore, \( \lim _{t\to \xi^- } c(t) \) exists. We
already know that \( b\rightarrow 0 \) as \(  t\searrow \xi .\) If \( a(t) \) approached a
finite limit as $t$ approached \( \xi ,\) then all three functions could be continuously
extended to \( \xi, \)  contradicting  the maximality of the interval of
existence. Therefore,
\[ \lim _{t\to \xi^- } a(t)= \infty . \]
Since \( a\rightarrow \infty ,\) but \( ac \) approaches a finite limit, it must be that
\( \lim _{t\rightarrow \xi } c(t) = 0. \) It follows, then, that as \( t\searrow \xi ,\) the
system can be approximated by
\begin{eqnarray*}
  a' & = & \frac{a}{2} (-a^{2} ) \\
  b' & = & \frac{b}{2} (a^{2} ) \\
  c' & = & \frac{c}{2} (a^{2} +2e^{-A} B) ,
\end{eqnarray*}
where \( B = \lim _{t\rightarrow \xi } ab ,\) a non-negative number. This system can be solved
explicitly. Solving, %the first equation,
we find that
\begin{align*} a(t) &\simeq \frac{1}{\sqrt{t-\xi } } ,\\
%(which is consistent with our earlier finding that \( a\rightarrow \infty \) as \( t\rightarrow \xi \) %).  Substituting this into each of the other two equations allows us to solve and find that
 b(t) &\simeq C\sqrt{t-\xi } ,\\
%and
 c(t) &\simeq D \sqrt{t-\xi } ,
\end{align*}
for constants $C$ and $D$ whose values do not affect the completeness question.
Regarding that
question, we recall that the  metric is of the form
\[ g = (abc)^{2} dt^{2} + a^{2} \sigma _{1} ^{2} + b^{2} \sigma _{2} ^{2} + c^{2} \sigma _{3} ^{2} . \]
Let
\( \gamma (s) \) be a curve that is constant in the orbit direction, and with  \( t(s) = s. \) Then
\[ l(\gamma ) = \lim _{\varepsilon \rightarrow 0} \int _{\xi + \varepsilon } ^{t_{1} } (abc) (s) \ ds < \infty  \]
by direct calculation. Since the length of this curve is finite, the distance to the boundary at $t=\xi$ is also finite, and the metric is incomplete.

{\bf (ii)} There is a point \( u\in (\xi , \eta ) \) at which \( c'(u) < 0. \) Then
\( a^{2} + b^{2} +2e^{-A} ab -c^{2} < 0 \) at that point. Similarly to above, we calculate the
derivative of this quantity:
\[ \frac{d}{dt} (a^{2} + b^{2} +2e^{-A} ab -c^{2} ) = c^{4} - (a^{2} -b^{2} )^{2} . \]
At $u,$ we have
\[ c^{2} > a^{2} +b^{2} + 2e^{-A} ab > a^{2} +b^{2} , \]
and therefore,
\[ c^{4} > (a^{2} + b^{2} )^{2} \geq (a^{2} -b^{2} )^{2} . \]
We therefore see that the derivative of this quantity is positive, implying that
\( c' (t) <0 \) on the entire interval \( (\xi ,u). \) We also find that \( c''(t) > 0 \) on
this interval. It must be that \( c \rightarrow \infty \) and \( a, \, b \rightarrow 0 \) as
\( t\rightarrow \xi .\) Permuting the roles of \( a, \, b, \) and $c$ in the approximated
equations that appeared in Case (i), we again find that the metric is incomplete.

\end{proof}

\subsection{Maximal solution intervals of the form $(-\infty,\eta)$}
We turn now to those trajectories for which \( \xi = -\infty .\) Analyzing the behavior of these
solutions as \( t\rightarrow -\infty ,\) we find which equilibrium points these approach.
%This is proved first, and then the question of completeness is addressed as a separate step.

\begin{prop} A trajectory for which \( \xi = -\infty \) converges to an equilibrium solution
of the form \( (q, q, 0), \) with \( q> 0, \) or \( ( q, 0, q) \) as \( t\rightarrow -\infty .\)
\end{prop}

\begin{proof} Since $b$ is increasing, its
limit as \( t\rightarrow -\infty \) exists, and so again the behavior of $b$ gives a
convenient way to split into cases.

{\bf (1)} Suppose first that \( \lim _{t\rightarrow -\infty } b(t) >0. \) In this case, and
because $ab$ and $bc$ are also increasing functions, $a$ and $c$ also have finite limits as
\( t\rightarrow -\infty .\) Comparing to the list of possible equilibrium solutions, and
remembering that we have assumed without loss of generality that \( a(t) \geq b(t), \) we see
that \( (a(t), b(t), c(t) ) \rightarrow (q, q ,0), \) with \( q >0. \) This
implies that
\[ \lim _{t\rightarrow -\infty } \frac{a(t)}{b(t)} =1. \]
From the earlier lemma,
\[ \frac{d}{dt} \Big(\frac{a}{b} \Big) = \frac{a}{b} (-a^{2} +b^{2} ) .\]
This is non-positive under the assumption that \( a \geq b, \) and it also follows that
\( a(t) /b(t) \geq 1 \) and is either strictly decreasing or else identically equal
to one. It can only be that \( a(t) \equiv b(t) \) for all \( t\in (-\infty , \eta ). \)

{\bf (2)} Now suppose that \( \lim _{t\rightarrow -\infty } b(t) =0. \) There are several
possibilities to consider.

{\bf (i)} Suppose there exists \( u \in (-\infty , \eta ) \) at which \( a'(u) <0. \) The
earlier calculation showed that at any point or on any interval where \( -a^{2} +b^{2} +c^{2} <0, \) that the derivative of this quantity is positive, and therefore \( -a^{2} +b^{2} +c^{2} \) remains
negative on all of \( (-\infty , u). \) Then $a$ is decreasing on all of \( (-\infty ,u), \)
and since at most one of the three functions can decrease on an interval, it follows that $b$ and
$c$ are non-decreasing. Those same calculations also give us that \( a''(t) >0 \) on this
interval and that
\[ \lim _{t\rightarrow -\infty } a(t) = \infty .\]
Since both \( \lim _{t\rightarrow -\infty } c(t) \) and \( \lim _{t\rightarrow -\infty } ac \)
are finite numbers, it must be that \( \lim _{t\rightarrow -\infty } c(t) = 0. \) For large,
negative values of $t$ therefore, the first equation can be approximated in the asymptotic
sense by
\[ a' = \frac{a}{2} (-a^{2} ) .\]
By a direct calculation, $a$ diverges at a finite value in the interval of existence,
which is a contradiction.

{\bf (ii)} The second possibility is that $c$ decreases at some point \( u\in (-\infty ,\eta ). \) As calculated previously, if \( c'(u) > 0, \) that is, if \( a^{2} + b^{2} +2e^{-A} ab -c^{2} < 0 \) at \( u\in (-\infty ,\eta ), \) then the derivative of this quantity is positive, and so $c$
remains a decreasing function on all of \( (-\infty ,u ). \) Also, \( c''(t) > 0 \) on this
interval, and so \( \lim _{t\rightarrow -\infty } c(t) = \infty . \) Because the product $ac$
has a finite limit, we have \( \lim _{t\rightarrow -\infty } a(t) =0. \) The third equation can be approximated for negative values of $t$ of large magnitude by
\[ c' = \frac{c}{2} (-c^{2} ) , \]
implying that \( c(t) \rightarrow \infty \) at an interior point, a contradiction.

{\bf (iii)} If all three of \( a, \, b, \) and $c$ are increasing on all of \( (-\infty , \eta ), \) then all three have finite limits as \( t\rightarrow -\infty , \) so \( (a(t), \, b(t), c(t) ) \)
converges to an equilibrium solution. Checking the list, the equilibrium solution is
\( (q, 0, q ) ,\) with the possibility \( q= 0\) not excluded.

\end{proof}

\subsection{Equilibrium $\mathbf{(q,q,0)}$, $\mathbf{q>0}$}
We now investigate the completeness of the metrics that correspond to trajectories converging to
equilibrium solutions of the form \( (q, q ,0) \)
%and \( (q, 0, q ), \) addressing especially the completeness question.

So assume the initial value was chosen to lie on the trajectory of an unstable curve of \( (q, q ,0), \)
with \( q >0. \)

\subsubsection{The endpoint $\xi=-\infty$}

For large, negative values of $t,$ both $a$ and $b$ can be approximated by \( q ,\) but for
\( c(t) ,\) we examine the equation itself, in order to determine the rate at which \( c(t) \rightarrow 0.\) Since \( c^{2} \) is small compared to $c,$ the third equation can be approximated
by
\[ c' \simeq (1+e^{-A} ) q ^{2} c. \]
For the moment, we will write \( \gamma = (1+e^{-A} ) q ^{2} , \) a positive constant. Solving the
equation gives
\[ c(t) \simeq k e^{\gamma t} , \]
where $k$ is a further constant. The metric can then be approximated as \( t\rightarrow -\infty \) by
\[ g = q ^{4} k^{2} e^{2\gamma t} dt^{2} + q ^{2} \sigma _{1} ^{2} + q ^{2} \sigma _{2} ^{2} + k^{2} e^{2\gamma t} \sigma _{3} ^{2} . \]
Making the change of variables
\[ y(t) = \frac{q ^{2} k}{\gamma } e^{\gamma t} ,\]
as in [DS1], the metric can be written as
\[ g = dy^{2} + q ^{2} \sigma _{1} ^{2} + q^{2} \sigma _{2} ^{2} + \frac{\gamma ^{2} }{q^{4} } y^{2} \sigma _{3} ^{2} , \]
or
\[ g = dy^{2} +q^{2} \sigma _{2} ^{2} + q^{2} \sigma _{2} ^{2} + (1+e^{-A} )^2 y^{2} \sigma _{3} ^{2} , \]
where \( t\rightarrow -\infty \) if and only if \( y\rightarrow 0. \)
Thus the distance to the boundary corresponding to $\xi=-\infty$ is finite.
However, in this case, under certain conditions that metric and K\"ahler form extend
smoothly to a singular orbit (a bolt). We now show this.

%{\bf Question.} Is the criterion for completeness the same as in Dancer and Strachan?

%If so, then the criterion is that \( 1+ e^{-A} \) be a half-integer. Remembering where the
%constant $A$ came from, we had \( e^{A} = (ab)/\alpha .\) So the criterion is that
%\( 1+ \alpha /(ab) \) be a half-integer. Since this quotient is constant along the trajectory,
%either the initial value or the equilibrium solution can be chosen in such a way that the
%condition be satisfied.

\subsubsection{Attaching a singular orbit at $\xi=-\infty$ for equilibrium $(q,q,0)$}

Denoting $r=\int_{-\infty}^ra(s)b(s)c(s)\,ds$ the metric can be transformed to the form
$g=dr^2+a^2\sig_1^2+b^2\sig_2^2+c^2\sig_3^2$.  Recall that the solutions we are examining
satisfy $a=b$ with equilibrium $(q,q,0)$, $q\ne 0$. In that case the Lie algebra
of Killing vector fields is one dimension higher, and the metric is preserved by a corresponding action of $U(2)$. The manifold then has the form $M=U(2)\times_{K}\mathbb{R}^2\setminus\{ 0 \}$,
with $K=U(1)\times U(1)$, and we wish to examine whether it is possible to attach smoothly a singular orbit (so-called bolt) at $r=0$, where the finite distance end of the manifold resides.
As the principal orbits are $3$-dimensional, the isotropy group is $H=U(1)$, and $K$ acts
on $K/H\approx S^1$ non-effectively. However one of its factors acts, of course, effectively,
and since $H$ embeds as the other factor, for the purposes of examining smooth extendibility,
one can equally consider just the action of $SU(2)\subset U(2)$, which is still of cohomogeneity one, and regard the manifold as $SU(2)\times_{U(1)}\mathbb{R}^2\setminus\{ 0 \}$. We will take this point of view in what follows. Note that $U(1)$ acts on $\mathbb{R}^2$ by a restriction of one the representations of $SU(2)$.

With this change of variables, equations \eqref{SU2-ode} in that case take the form
\begin{align*}
\fr{da}{dr}&=\fr c{2a},\\
%\fr{db}{dr}&=\fr b2\Big(\fr a{bc}-\fr b{ac}+\fr c{ab}\Big),\\
\fr{dc}{dr}&=1+e^{-A}-\fr{c^2}{2a^2}.\\
\end{align*}
We see from these equations that $a$ can be smoothly extended as an even function and $c$ as an odd
one near $r=0$. Using the notations in \cite{vz}, we denote
\[
\mathfrak{k}=\mathrm{span}(X_3),\qquad \mathfrak{m}=\mathrm{span}(X_1,X_2)=\ell_1\qquad
V=\mathrm{span}(\partial_r,X_3):=\ell_{-1}',
\]
where $X_i, i=1\ldots 3$ respectively denote a usual the dual vectors to $\sig_i$.

Consider now a solution $(a,b,c)$ approaching as $r\to 0$ the equilibrium point $(q,q,0)$, with $a=b$. As in $SU(2)$, with our choice of normalization of the Lie algebra basis, we have $\exp(4\pi X_3)=\mathrm{id}$, $2X_3$ generates an $U(1)$-action
whose rotational isotropy action on the plane spanned by it and $d/dr$ is by $a_1\theta$, and on $\mathfrak{m}$ by
$d_1\theta$, where the constants $a_1$ and $d_1$ are determined as follows. First,
\begin{align*}
a_1&=\lim_{r\to 0}\fr{|2X_3|}r=\lim_{r\to 0}\fr {2c}{r}=\lim_{r\to 0}2\fr {dc}{dr}\\
&=2\lim_{r\to 0}\Big(1+e^{-A}-\fr{c^2}{2a^2}\Big)=2(1+e^{-A}-\fr 0{2q^2})=2(1+e^{-A}).\\[2pt]
\end{align*}
Since $c^2(r)$ is even with no constant term, following \cite{vz},  smooth extendibility requires first of all that $2(1+e^{-A})$ is an integer.

Second, as $[X_3,X_1]= X_2$ and $[X_3,X_2]=- X_1$, we have $d_1=1$.
The remaining potentially nontrivial smoothness conditions for the metric in \cite{vz} are
\begin{align*}
a^2+b^2&=\phi_1(r^2),\\
a^2-b^2&=r^{\fr{2d_1}{a_1}}\phi_2(r^2),
\end{align*}
for some functions $\phi_1$, $\phi_2$. The first of these clearly holds as $a=b$
can be extended to an even function. The second is also obvious as $a=b$.
%For the second, we note from the equations that
%$a^2-b^2=c(a\fr{db}{dr}-b\fr{da}{dr})$, so that writing the expansions $a=q+\al r^2+\ldots$,
%$b=q+\beta r^2+\ldots$, $c=a_1r+\gamma r^3+\ldots$, one easily sees that
%$c(a\fr{db}{dr}-b\fr{da}{dr})=2qa_1(\beta-\alpha)r^2+\ldots$, so that the metric extends
%smoothly to a bolt at $r=0$ if and only if $A=0$.

We turn to checking that the K\"ahler form extends across the singular orbit at $r=0$.
For the isotropic action of $SO(2)$ on $T_pM$, we find that $\partial_r+\frac{i}{r}X_3$ is an eigenvector with eigenvalue $e^{ia_1\theta}$, and $X_1+iX_2$ is an eigenvector with eigenvalue $e^{id_1\theta}$, and likewise for their complex conjugates.
Dualizing gives eigenspaces of $T^*_pM$: $dr-ir \sigma_3$ has eigenvalue $e^{ia_1\theta}$, and $\sigma_1-i\sigma_2$ has eigenvalue  $e^{id_1\theta}$.
Thus the eigenspaces of $\Lambda^2T^*_pM$ are
\begin{align*}
E_1 &= \mathrm{span}\{r dr\wedge \sigma_3, \sigma_1\wedge\sigma_2\}\\
E_{e^{i(a_1+d_1)\theta}}&=\mathrm{span}\{dr\wedge\sigma_1-r\sigma_3\wedge\sigma_2
+i(-dr\wedge\sigma_2-r\sigma_3\wedge\sigma_1)\}\\
E_{e^{i(a_1-d_1)\theta}}&=\mathrm{span}\{dr\wedge\sigma_1+r\sigma_3\wedge\sigma_2
+i(dr\wedge\sigma_2-r\sigma_3\wedge\sigma_1)\}\\
&\qquad
\end{align*}
The smoothness condition is the equivariance condition $\omega(e^{a_1\theta}p)=\exp(\theta X_3)^*\omega$.
This requires that the coefficient of
\begin{align}
E_1 & \text{ is } \phi_1(r^2), \\
E_{e^{\pm i(a_1-d_1)\theta}}& \text{ is } r^{\frac{|a_1-d_1|}{a_1}}\phi_2(r^2),\\
E_{e^{\pm i(a_1+d_1)\theta}}& \text{ is } r^{\frac{|a_1+d_1|}{a_1}}\phi_3(r^2).
\end{align}
Now we have
\begin{align*}\omega&=c\,dr\wedge\sigma_3+ab\sigma_1\wedge\sigma_2\\
&=\frac{c}{r}\cdot rdr\wedge\sigma_3+a^2\sig_1\we \sig_2.
\end{align*}
Thus our only nonzero coefficients are in $E_1$. Thus the smoothness conditions become
\begin{align*}
\fr cr &= \phi_1(r^2), \\
a^2 &= \phi_2(r^2). \\
\end{align*}
Both of these hold trivially from the even/odd extendibility of $a$, $c$. Thus in total,
the K\"ahler form extends smoothly across the singular orbit if and only $2e^{-A}$ is an integer.

\subsubsection{The endpoint $\eta$}
Next, we consider the question of completeness as \( t\rightarrow \eta \).% must also be resolved.
%The argument is similar to that of [DS1].

Recalling here that in this case that \( a \equiv b, \) the system reduces to
\begin{align}
 a' & =  \frac{a}{2} c^{2}\nonumber \\
 c' & =  \frac{c}{2} (2(1+e^{-A} ) a^{2} -c^{2} ) .\lb{a=b-case}
\end{align}

Following \cite{d-s1}, one can write the equations using the variables $w_1=bc$, $w_2=ac$, $w_3=ab$.
In the case at hand $w_1=w_2$, and the above two equations are then equivalent to the system
\begin{align}
 w_{1} ' & =   (1+e^{-A} ) w_{1} w_{3},\nonumber \\
 w_{3} ' & =  w_{1} ^{2} ,\lb{2d-sys}
\end{align}
with the metric given in terms of \( w_{1} \) and \( w_{3} \) as
\[ g = w_{1} ^{2} w_{3} dt^{2} + w_{3} \sigma _{1} ^{2} + w_{3} \sigma _{2} ^{2} + \frac{w_{1} ^{2} }{w_{3} } \sigma _{3} ^{2} . \]
A curve \( \gamma :[t_{0} ,\eta ) \rightarrow M \) which is constant in the orbit direction has
length
\[ l(\gamma ) = \int _{t_{0} } ^{\eta } w_{1} \sqrt{w_{3}} \ dt . \]
In order to determine the behavior of \( w_{1} \) and \( w_{3} \) as \( t\rightarrow \eta , \) \
first eliminate \( w_{1} \) to reduce to an equation in \( w_{3} \) alone, which integrated once gives
\be\lb{w3} w_{3} ' (t) = (1+e^{-A} ) w_{3} ^{2} (t) + \delta , \end{equation}
where \( \delta \) is a constant of integration.

\begin{lemma}
$\eta $ is finite.
\end{lemma}

\begin{proof}
First observe that \( w_{1} ', \, w_{3} ', \, w_{3} '' >0. \) Next, we see that \( w_{1} \) and
\( w_{3} \) become unbounded as \( t\rightarrow \eta ,\) whether \( \eta \) is finite or  not. If
\( \eta = \infty ,\) then since \( w_{3} ' \) and \( w_{3} '' \) are both positive, $w_{3} $ must
become unbounded; since \( w_{3} ' = w_{1} ^{2} ,\) then \( w_{1} \rightarrow \infty \) also. If
\( \eta < \infty ,\) then at  least one of \( w_{1} , \, w_{3} \) becomes unbounded, because otherwise the maximality of the interval of existence would be contradicted. From the equations,
if one becomes unbounded, so does the other.

A comparison argument now shows that in fact \( \eta < \infty .\) Since \( w_{3} \rightarrow \infty \) as
\( t \rightarrow \eta , \) there exists an $M$ so that for all \( t\in (M, \eta ), \)
\[ (1+e^{-A} )w_{3} ^{2} (t) + \delta > w_{3} ^{2} (t) .\]
Choose \( t_{1} \in (M, \eta ), \) and consider the comparison equation
\[ w' = w^{2} , \]
with initial value \( (t_{1} ,w_{3} (t_{1} ) ). \) The actual solution \( w_{3} \) passes through the point \( (t_{1} , w_{3} (t_{1} ) ) \) and has a steeper derivative at every point, so \( w_{3} \) lies above the comparison function $w.$ The comparison function becomes unbounded at a finite
value \( \eta _{c} ; \) its solution is \( w = 1/(\eta _{c} -t) .\) Therefore, \( w_{3} \) also
diverges to $\infty$ at some \( \eta \leq \eta _{c} < \infty .\)

\end{proof}

Since \( w_{3} \rightarrow \infty \) as
\( t\rightarrow \eta ,\) a multiple of that same comparison function also serves in the asymptotic sense:
\[ \lim _{t\rightarrow \eta } \frac{(1+e^{-A} )w_{3} ^{2} (t) +\delta }{(1+e^{-A} )w_{3} ^{2} (t) } = 1. \]
The behavior of \( w_{3} \) can be discerned from that of $w,$ the solution to \( w' = (1+e^{-A} )w^{2} .\) This solution is a multiple of \( 1/(\eta -t) .\) Using the equation \( w_{3} ' = w_{1} ^{2} , \) we see that
\[ \int _{t_{1} } ^{\eta } w_{1} \sqrt{w_{3} } \  dt =\infty, \]
since the integrand is asymptotically $(t-\eta)^{-3/2}$.
Thus the metric is complete.

\subsection{Equilibrium $\mathbf{(0,0,0)}$}

\subsubsection{Reduction to an explicit solution}
The case of equilibrium $(0,0,0)$ contains many of the ideas already introduced, so we will
be brief. First, in order to find the rates of convergence of $a$, $b$, $c$ as $t\to-\infty$,
convert the system \Ref{SU2-ode} from the variable $t$ to the variable $r=\int_{-\infty}^tabc\,ds$. Solutions approaching this equilibrium have that $a$, $b$, $c$
can be extended smoothly as odd functions of $r$. Writing odd power series expansions of $a$, $b$, $c$ in $r$ and solving for the first order coefficients $\hat{a}$, $\hat{b}$, $\hat{c}$ respectively, yields positive solutions $\hat{a}=\hat{b}=\sqrt{\gamma}/2$, $\hat{c}=\gamma/2$ for $\gamma:=1+e^{-A}$. Hence $a/b$ tends to $1$ as $r\searrow 0$ (or $t\to-\infty$), and as
before we must have $a\equiv b$.

Thus in this case the ODE system again becomes \Ref{2d-sys} and equation \Ref{w3}
also holds. Now since $w_3'=w_1^2$, $w_3$ is an increasing function converging to $0$ as $t\to-\infty$. Therefore $w_3'$ converges to $0$ as $t\to-\infty$ and thus $w_3'-\gamma w_3^2\to0-\gamma 0=0$, so that we have $\delta=0$. The ODE system thus simplifies and its
solution is just a case of the one used as comparison in the previous subsection, specifically
$a(t)=\sqrt{\fr1{-\gamma t}}$, $c(t)=\sqrt{\fr{1}{-t}}$.

%\subsubsection{The endpoint $\eta$} The analysis of the previous section even more easily in this %explicit case, yields that the endpoint $\eta$ is infinitely far.

\subsubsection{No smooth extension to a singular orbit}
The distance to the endpoint $t=-\infty$ in this explicit metric is finite.
Hence one needs to investigate whether the metric and K\"ahler form can be extended smoothly to a singular orbit. The singular orbit in this case is just a point, i.e. a ``nut".
But in the $r$ coordinate one easily sees that $a(r)=\sqrt{\gamma}r/2$, $c(r)=\gamma r/2$, and since we cannot have both $a'(0)=1$ and $c'(0)=1$ simultaneously, the metric is not complete. This can alternatively be deduced from the statement of the smoothness condition in \cite{vz} for the
case where the isotropy subgroups of a singular fiber is $Sp(1)$ and generic isotropy subgroup is trivial.

%This time $-2X_3$ generates rotations by angle $a_3\theta$ on $\mathrm{span}(d/dr,-2X_3)$, where %$a_3$ is computed now to be $\gamma$, whereas $-2X_2$ generates rotations by angle $a_2\theta$
%on $\mathrm{span}(d/dr,-2X_2)$, where a similar computation gives $a_2=\sqrt{\gamma}$. Since in %the terminology of \cite{vz} $\mathfrak{m}$ is trivial, these two facts are the only ones
%pertinent to smoothness of the metric, and together they yield the smoothness condition that
%$\gamma=1+e^{-A}$ must be a square of an integer.

\subsection{Equilibrium $\mathbf{(q,0,q)}$, $\mathbf{q>0}$}

Suppose now that the initial value is chosen to lie on an unstable curve of \( (q , 0, q ), \) with \( q > 0. \)

For large, negative values of $t,$ the functions $a$ and $c$ can be approximated by \( q ; \) to discover the rate of vanishing of $b,$ we examine the equation
\[ b' = q ^{2} b. \]
This is easily solved to obtain \( b(t) = k e^{q ^{2} t} .\) After a change of variables
\( v(t) = k e^{q^{2} t}  \) (the same as in [DS]), the metric is approximated by
\[ g = dv^{2} + q ^{2} \sigma _{1} ^{2} + v^{2} \sigma _{2} ^{2} + q ^{2} \sigma _{3} ^{2} , \]
with \( t\rightarrow -\infty \) if and only if \( v\rightarrow 0. \) This metric is incomplete.
However, in this case we cannot extend the metric smoothly to a singular orbit. Namely,
upon switching to the coordinate $r$, so that the metric has the form $dr^2+h_r$, in the resulting
ODE system $c$ can be extended only as an odd function near $r=0$ (whereas $a$ and $b$ can be extended either both as even, or both as odd functions). But an odd function can't have a nonzero value $q$ at
$r=0$. Therefore the corresponding metric is necessarily incomplete, and we do not pursue this
case further.

Collecting these investigations, we summarize the findings.

\begin{thm}\lb{thm3}
Let $(M,g)$ be a Riemannian $4$-manifold admitting a cohomogeneity one $SU(2)$-action by isometries. Then $g$ is a complete diagonal centrally flat K\"{a}hler metric precisely when
it is of the form \Ref{bianchiAmet} with $(a, b, c)$ an unstable solution curve of the system \Ref{SU2-ode} for an equilibrium point $(q,q,0)$, $q>0$, defined on a maximal interval. Such metrics satisfy $a=b$, and $M$ contains a unique singular orbit.
\end{thm}

With regard to the explicitness of these solutions, note that one could have proceeded
with the system \Ref{a=b-case} by making the change of variables as in section \ref{sec:heis}.
This would give a similar explicit solution, where this time
\[
\phi(q)=\fr{e^k}{2\gamma}e^{2\gamma q}+B, \text{ with $\gamma=1+e^{-A}$ and constants $k$, $B$.}
\]
The case of positive $B$ corresponds to a solution converging to equilibrium $(q,q,0)$, $q>0$, whereas $B=0$ yields one converging to $(0,0,0)$.

\section{Centrally flat metrics under the Euclidean Group of plane motions}\lb{sec:Euc}%with $p_3\ne 0$}
In this section we describe a complete triaxial centrally flat
metric with a cohomogeneity one action of the Euclidean group $E(2)$.
The method employed is that of the recent \cite{mr}, which in turn
was inspired by \cite{d-s1}.

% and \cite{d-s2}
%which gives an analogous result for the case of the compact group
%$SU(2)$. Two main differences in method from those references are
%a systematic use of recent results of Verdiani and Ziller \cite{vz},
%and the establishment of Cauchy-Schwarz type estimates yielding
%completeness for this non-compact group.

We set $p_2=0$, $p_1=p_3=1$, and $\lambda=0$.
Then the Lie algebra spanned by $X_1,X_2,X_3$ is the Lie algebra of the Euclidean group.
The equations for zero central curvature  are, from \Ref{K1}-\Ref{KKE} and  \Ref{alpha}
\begin{align}
\label{E2ODEa}a'&=\frac{a}{2}(-a^2+c^2), \\
\label{E2ODEb}b'&=\frac{b}{2}(a^2+c^2), \\
\label{E2ODEc}c'&=\frac{c}{2}(a^2-c^2+2\al),\\
\lb{E2ODEd} \al'&=\al c^2.
\end{align}
These can be reduced to a system of three equations as in section \ref{sec:SU2}, but we will
generally stick with the above version.

As in the case of $SU(2)$, the derivatives in this system are given by polynomials in the dependent variables, hence are locally Lipschitz, so that standard ODE theory applies. The symmetries of these equations include, as they are autonomous, constant shifts in $t$. Additionally, the equations possess a scaling symmetry
\[
(a(t),b(t),c(t),\al(t))\to (ka(k^2t),b(k^2t),kc(k^2t),k^2\al(k^2t)),
\]
taking solutions to solutions.

\subsection{Linearization about Equilibria}
The equilibrium solutions are $(q,0,q,0)$ and $(0,p,0,r)$, and we concentrate
on the nonzero case.
Then the system \Ref{E2ODEa}-\Ref{E2ODEd}
has linearization about $(q,0,q,0)$ given by
\begin{align*}
a'&=-q^2a+q^2c, \\
b'&=q^2b, \\
c'&=q^2a-q^2c+q\al, \\
\al'&=q^2\al,
\end{align*}
which has one double positive, one negative and one zero eigenvalue for $q> 0$. The linearization about $(0,p,0,r)$ has three zero eigenvalues and one with the sign of $r$.

\begin{thm}\lb{thm4}
A solution of \Ref{E2ODEa}-\Ref{E2ODEd} yields a complete centrally flat metric
of the form \Ref{bianchiAmet} on a cohomogeneity one $E(2)$ $4$-manifold if it
is a solution along an unstable curve of an equilibrium point $(q,0,q,0)$, $q>0$.
\end{thm}
\begin{proof}
The proof is broken into three steps.
As in the case of $SU(2)$, solutions with a maximal interval
having a finite left endpoint do not yield complete metrics. See
Proposition~\ref{prop:solutions}. Solutions with maximal interval $(-\infty,\eta)$ are the unstable curves of the equilibrium points $(q,0,q,0)$, and satisfy $0\le c^2-a^2 \le 2\al$.
Once again for a geodesic orthogonal to the orbits, $\eta$ is infinitely far, while
$t=-\infty$ is at a finite distance. See Proposition~\ref{prop:estimates}. At $t=-\infty$
the metric nd K\"ahler form extend smoothly (Proposition~\ref{prop:bolt}). The proof that all finite length curves remain inside some compact set is as in \cite{mr}.
%To complete the proof that this gives us a complete centrally flat metric we show that $g$ and %$\omega$ can be extended smoothly as $t\to-\infty$, and then finish the proof of completeness by
%showing all finite length curves remain inside some compact set.
%These are completed in Propositions~\ref{prop:bolt} and \ref{prop:complete}, respectively.
\end{proof}

We first record in a lemma some relations,  easily verifiable via \Ref{E2ODEa}-\Ref{E2ODEc}, which will
be used later in the proof.
%We give the lemma in a form valid for any unimodular group.
\begin{lemma}\lb{prelim}
For the system \Ref{E2ODEa}-\Ref{E2ODEd},
\begin{align*}
(ab)'&=abc^2,\qquad (\al c)'=\fr{\al c (a^2+c^2+2\al)}2, \\
(bc)'&=bc\left(a^2+\al\right),\qquad  \left(\frac{a}{b}\right)'=-\frac{a^3}{b},\\
(ac)'&=ac\al, \\
%(\al c)'&=\fr{\al (c^3+ca^2)}2+2\al^2,\\
%\left(\frac{a}{b}\right)'&=-\frac{a^3}{b}, \\
%(a^2)'&=(a^2)(-a^2+c^2) \\
%(a^2-c^2)'&=-(a^2-c^2)(a^2+c^2)-2 a^2b^2c^2 \\
%-(c^2)'&=-c^2(a^2-c^2)-2 a^2b^2c^2
\end{align*}
%\begin{align*}
%(ab)'&=p_3abc^2, \\
%(bc)'&=bca^2\left(p_1-\frac{\lambda}{p_3}b^2\right), \\
%(ac)'&=acb^2\left(p_2-\frac{\lambda}{p_3}a^2\right), \\
%\left(\frac{a}{b}\right)'&=-\frac{a}{b}(p_1a^2-p_2b^2), \\
%(p_1a^2-p_2b^2)'&=(p_1a^2-p_2b^2)(-p_1a^2-p_2b^2+p_3c^2) \\
%(p_1a^2-p_3c^2)'&=(p_1a^2-p_3c^2)(-p_1a^2+p_2b^2-p_3c^2)+2\lambda a^2b^2c^2 \\
%(p_2b^2-p_3c^2)'&=(p_2b^2-p_3c^2)(p_1a^2-p_2b^2-p_3c^2)+2\lambda a^2b^2c^2
%\end{align*}
\end{lemma}

\subsection{Solutions}
\begin{prop}\label{prop:solutions}
There are no complete metrics corresponding to solutions of \Ref{E2ODEa}-\Ref{E2ODEd} with maximal interval $(\xi,\eta)$, when $\xi$ is finite. Furthermore, the unstable curves of the equilibrium points $(q,0,q,0)$ are non-equilibrium solutions with maximal interval $(-\infty,\eta)$ which satisfy $0\le c^2-a^2\le 2\al$.
\end{prop}
\begin{proof}
For an initial time $t_0$, let $(\xi,\eta)$ be a maximal solution interval for the initial value problem for \Ref{E2ODEa}-\Ref{E2ODEd} with
$a(t_0)=a_0$, $b(t_0)=b_0$, $c(t_0)=c_0$ and $\al(t_0)=\al_0$.

Uniqueness of solutions to \Ref{E2ODEa}-\Ref{E2ODEd} implies that if any of
$a$, $b$, $c$ or $\al$ are zero anywhere in $(\xi,\eta)$ then they are zero everywhere.
Accordingly we assume that $a$, $b$, $c$ and $\al$ are all positive on $(\xi,\eta)$.
Then we see from Lemma~\ref{prelim} and \Ref{E2ODEb} that $ab$, $bc$, $ac$, and $b$ are all
increasing on $(\xi,\eta)$.

We consider the following cases:
\subsubsection*{Case 1: $c_0^2-a_0^2<0$} We first make the following claim.\\[4pt]
Claim: In this case $a\to\infty$ as $t\to\xi^+$.\\[3pt]
\textit{Proof of claim:} Since
\[ (c^2-a^2)'=-(c^2-a^2)(c^2+a^2)+2\al c^2, \]
if $c^2-a^2<0$ then $(c^2-a^2)'>0$, thus $c^2-a^2<0$ for all $\xi<t<t_0$.
Therefore,
\begin{align*}
a' &= \frac{a}{2}(c^2-a^2), \\
a'' &= \frac{a}{4}[(c^2-a^2)^2-2(c^2-a^2)(c^2+a^2)+4\al c^2],
\end{align*}
showing that $a$ is decreasing and concave up on $(\xi,t_0)$.
Next, we always have $b'>0$, while on $(\xi,t_0)$
\[c'=\frac{c}{2}(a^2-c^2+2\al)>0, \]
i.e. $c$ is increasing on $(\xi,t_0)$.
Therefore $b$ and $c$ are bounded on $(\xi,t_0)$.
Thus, as $(\xi,\eta)$ is the maximal solution interval,
$a$ could be bounded as $t\to\xi^+$ only if $\xi=-\infty$.
But since $a$ is concave up, $a\to\infty$ as $t\to\xi^+$ even when $\xi=-\infty$. \qed

Since $ab$ and $ac$ are increasing, they are bounded as $t\to\xi^+$ and $a\to\infty$, so $b\to0$, $c\to0$. Now $\al$ is also increasing, so $\al\to k$ for some constant $k$ as $t\to\xi^+$.
Then as $t\to\xi^+$ the first three equations will take the asymptotic form
\begin{align*}
a'&=-\frac{1}{2}a^3 \\
b'&=\frac{1}{2}ba^2 \\
c'&=\frac{1}{2}c(a^2+2k) \\
\end{align*}
the solution of which has asymptotic form
\begin{align*}
a&\simeq (t-\xi)^{-\frac{1}{2}},\\
b&\simeq b_1(t-\xi)^{\frac{1}{2}},\\
c&\simeq c_1(t-\xi)^{\frac{1}{2}},\\
\end{align*}
for some constants $b_1$ and $c_1$.
This shows that $\xi$ is finite in this case and
\[ \int_{\xi}^{t_0} abc\,dt <\infty,\]
so the metric is not complete.

\subsubsection*{Case 2: $c_0^2-a_0^2>2\al_0$} Here we have a similar claim.\\[4pt]
Claim: In this case $c\to\infty$ as $t\to\xi^+$.\\[3pt]
\textit{Proof of claim:} Analogous to the previous claim.\qed
%Since
%\[ (c^2-a^2-2\al)'=-(c^2-a^2)(c^2+a^2), \]
%if $c^2-a^2>0$ then $(c^2-a^2-2\al)'<0$, thus in this case, as $\al_0>0$, $c^2-a^2>2\al$ for all %$\xi<t\le t_0$.
%Therefore, as
%\begin{align*}
%c' &= \frac{c}{2}(a^2-c^2+2\al), \\
%c'' &= \frac{c}{4}[(a^2-c^2+2\al)^2+2(c^2-a^2)(c^2+a^2)],
%\end{align*}
%we see that $c$ is decreasing and concave up on $(\xi,t_0)$.
%Next, we always have $b'>0$, while
%$a$ is increasing on $(\xi,t_0)$.
%Therefore $a$ and $b$ are bounded on $(\xi,t_0)$.
%Thus, as $(\xi,\eta)$ is the maximal solution interval,
%$c$ could be bounded as $t\to\xi^+$ only if $\xi=-\infty$.
%But since $c$ is concave up, $c\to\infty$ as $t\to\xi^+$
%even when $\xi=-\infty$. \qed

Since $ac$, $bc$ and $\al c$ are increasing (see Lemma~\ref{prelim}),
they are bounded as $t\to\xi^+$ and $c\to\infty$, so $a\to0$, $b\to0$ and $\al\to 0$.
Then as $t\to\xi^+$ the equations take the asymptotic form
\begin{align*}
a'&=\frac{1}{2}ac^2 \\
b'&=\frac{1}{2}bc^2 \\
c'&=-\frac{1}{2}c^3 \\
\end{align*}
which has solution
\begin{align*}
a&\simeq a_1(t-\xi)^{\frac{1}{2}}\\
b&\simeq b_1(t-\xi)^{\frac{1}{2}}\\
c&\simeq (t-\xi)^{-\frac{1}{2}}\\
\end{align*}
for some constants $a_1$ and $b_1$.
This shows that $\xi$ is finite in this case and
\[ \int_{\xi}^{t_0} abc\,dt <\infty,\]
so the metric is not complete.

If $c^2-a^2<0$ or $c^2-a^2>2\al$ at any time, then a constant shift in $t$ will give one of the previous cases.
In both previous cases, $\xi$ is finite, but we know that the unstable curve of the equilibrium points $(q,0,q,0)$ must have $\xi=-\infty$.
The existence of these curves is guaranteed by the center manifold theorem.
Therefore we consider the final case:

\subsubsection*{Case 3: $0\le c^2-a^2 \le 2\al\textnormal{ for all }t\in(\xi,\eta)$} Here we
have a different claim.\\[4pt]
Claim: In this case $\xi=-\infty$. \\[3pt]
\textit{Proof of claim:}
In this case $a$, $b$, and $c$ are all increasing, therefore they are all bounded on $(\xi,t_0)$.
Since $(\xi,\eta)$ is the maximal solution interval $\xi=-\infty$. \qed

As $a$, $b$, $c$ and $\al$ are all increasing, it must be that they all approach finite non-negative limits as $t\to-\infty$.
Thus $(a,b,c,\al)$ must approach an equilibrium point.
If $(a,b,c,\al)\to(0,p,0,r)$ with $p>0$, then $a/b\to 0$ as $t\to-\infty$, but $a/b$ is decreasing and positive (see Lemma~\ref{prelim}), so this cannot happen. On the other hand, if $r>0$ and $p=0$,
note first that from \Ref{E2ODEa}-\Ref{E2ODEb} it easily follows that $\al=kab$ for some constant $k$ which is positive for a non-equilibrium solution. Then, as  $a/\al$ approaches $0$ as $t\to-\infty$, so does $1/b$, but $1/b$ approaches $\infty$, which is a contradiction.

Therefore, when $t\to-\infty$ we see that $(a,b,c,\al)\to(q,0,q,0)$.
\end{proof}

Note that we did not rule out the possibility that $q=0$. However, power series calculations
show at least that there are no non-equilibrium trajectories approaching $(0,0,0,0)$ which
are analytic, in an appropriate sense, at $t=-\infty$. From now on we will only consider the case $q>0$. The center manifold theorem guarantees that solutions exist and are defined over a maximal interval with left endpoint $-\infty$, while the above proof shows that $a$ and $c$, along with $b$ and $\al$ are non-decreasing on this interval.

We will need a one more property of the solutions in Case 3.
\begin{lemma}\lb{ab-unbd}
In Case 3 above, $ab$ is unbounded from above.
\end{lemma}
\begin{proof}
By \Ref{E2ODEc} and \Ref{E2ODEd}
\[
(c)'=\fr c2(a^2-c^2+2\al)\le \fr c2 2\al=\fr c2 \fr{2\al'}{c^2}=\fr{2\al'}{2c}
\]
so $(c^2)'\le 2\al'$ or $c^2|_s^t\le 2\al|_s^t$. Taking $s\to-\infty$ gives
\be\lb{al-c}
2\al\ge c^2-q^2.
\end{equation}
Applying this to \Ref{E2ODEc} gives
\[
(\log c^2)'\ge a^2-c^2+c^2-q^2=a^2-q^2.
\]
As one easily checks, as we are always assuming $\al$ is not identically zero, there is no
non-equilibrium solution with $a=q$ identically. Thus
for some $t_0$, for any $t>t_0$, $a(t)\ge a(t_0)>q$, so on that domain $(\log c^2)'>\epsilon>0$.
Thus $c^2$ grows faster than exponentially, and hence so does $2\al$ by \Ref{al-c}.
And $\al=kab$, $k>0$. This proves the result if $\eta=\infty$. If $\eta$ is finite, one of
$a$, $b$, $c$, $\al$ is unbounded and they are all increasing, which proves the claim if
it is $a$ or $b$ that are unbounded. If it is $c$, $\al$ is also unbounded by \Ref{al-c} again.
\end{proof}

%\begin{prop}\label{prop:estimates}
%Let $g$ be a metric of the form \Ref{bianchiAmet} on a manifold $M$, with $a$, $b$, $c$ a solution to %\Ref{E2ODEa}-\Ref{E2ODEc} along an unstable curve of an equilibrium point $(q,0,q)$, $q>0$.
%For any curve in $M$ $g$-orthogonal to the $G$-orbits containing a point $p$,
%% following curves along $t$,
%its endpoint corresponding to $\xi=-\infty$ is at finite distance, while its endpoint corresponding to %$\eta$ is at infinite distance, from $p$.
%\end{prop}
\begin{prop}\label{prop:estimates}
Let $g$ be a Riemannian metric of the form \Ref{bianchiAmet} on an $E(2)$-manifold $M$, with $a$, $b$, $c$ a solution to \Ref{E2ODEa}-\Ref{E2ODEd} along an unstable curve of an equilibrium point $(q,0,q,0)$,
$q>0$, having maximal domain $I=(-\infty,\eta)$. Assume that the latter interval is also the range
of the coordinate function $t$ on $M$. For a point $p_0\in M$ with orbit through $p_0$ of principal
type and a level set $M^t$ of $t$,
\[
\lim_{t\to -\infty}d_g(p_0,M^t)<\infty,\qquad
\lim_{t\to \eta}d_g(p_0,M^t)=\infty,
\]
where $d_g$ is the distance function induced by $g$.
%Suppose $\gamma: I\to M$ is a smooth injective map defining a
%curve in $M$ which passes through a point $p_0$, is $g$-orthogonal to the $G$-orbits and satisfies
%$t\circ\gamma=\mathrm{id}_I$.
%Then
%\[
%\lim_{t\to -\infty}d_g(p_0,\gamma(t))<\infty,\qquad
%\lim_{t\to \eta}d_g(p_0,\gamma(t))=\infty,
%\]
%where $d_g$ is the distance function induced by $g$.
\end{prop}
\begin{proof}
As in \cite{mr} we note that
%The union of the principal orbits forms an open dense set, $\tilde{M}$, so that
%$\tilde{M}/\mathcal{G}$ is a smooth manifold of dimension 1. The function $t$ is
%a smooth submersion from $\tilde{M}$ to $\tilde{M}/\mathcal{G}$. The metric
%\[ (abc)^2dt^2\]
%makes this into a Riemannian submersion.
the level sets of $t$ are orbits of $\mathcal{G}$ and for $t_0=t(p_0)$
\[d_g(p_0,M^{t_1}) = d_g(M^{t_0},M^{t_1}), \]
measures the distance in the quotient manifold $\tilde{M}/\mathcal{G}$, where
\be\lb{dist-def} d_g(M^{t_0},M^{t_1}) = \left|\int_{t_0}^{t_1} abc dt\right|, \end{equation}
and the metric is $(abc)^2dt^2$.

We omit the proof that $\lim_{t\to -\infty}d_g(p_0,M^t) < \infty$ as it is identical to that in \cite{mr}, and also similar to the case of $SU(2)$.
%Asymptotically as $t\to-\infty$,
%\begin{align*}
%a &\simeq q\\
%b &\simeq ke^{q^2t}\\
%c &\simeq q
%\end{align*}
%This gives the asymptotic metric
%\[ g\simeq k^2q^4e^{2q^2t}dt^2+q^2\sigma_1^2+k^2e^{2q^2t}\sigma_2^2+q^2\sigma_3^2, \]
%and for $v=ke^{q^2t}$ this is, similar to the $SU(2)$ case,
%\[ g\simeq (dv^2+v^2\sigma_2^2)+q^2(\sigma_1^2+\sigma_3^2). \]
%In this coordinate, the endpoint $\xi=-\infty$ is at $v=0$, and we see that
%\[ \lim_{t\to -\infty}d_g(p_0,M^t) = \int_{-\infty}^{t_0} abc dt = \int_0^{v_0} dv < \infty. \]
To understand the behavior at the $\eta$ side of the solution interval, we adopt the change of
variable $r=2(ab)^{1/2}$ first appearing in \cite{pp1} and \cite{d-s1}, which is allowable as $ab$ is strictly increasing (Lemma~\ref{prelim}). $r\to\infty$ as $t\to\eta$ since otherwise $ab$ is bounded, contradicting Lemma~\ref{ab-unbd}. Using Lemma~\ref{prelim}, the metric after this change
takes the form
\be\lb{gWV}
g=W^{-1}dr^2+\fr{r^2}4(V\sig_1^2+V^{-1}\sig_2^2+W\sig_3^2)
\end{equation}
with $W=c^2/(ab)$ and $V=a/b$. Additionally,
\begin{align*}
\fr{dW}{dr}=W'/r'&=\Big(\fr{c^2}{ab}\Big)'(ab)^{-1/2}c^{-2}\\
&=-\fr 4rW+\fr 2r\fr ab+16\fr{\al}{r^3}.
\end{align*}
Now $a/b$ decreases to a finite nonnegative limit $L$ as $r\to\infty$, so
that asymptotically
\[
\fr{dW}{dr}= -\fr 4rW+\fr 2rL+16\fr{\al}{r^3},
\]
an equation which, using the aforementioned relation $\al=kab=k\fr{r^2}4$, $k>0$ constant,
has solution
\[
W=L/2+k+\fr p{r^4}
\]
for an integration constant $p$.
%The metric has the form
%\[
%g=W^{-1}dr^2+\fr 14r^2(V\sig_1^2+V^{-1}\sig_2^2+W\sig_3^2),
%\]
The metric then has the asymptotic form \Ref{gWV} for $W$ as above and $V=L$. If $L=0$ this asymptotic
form is degenerate, but nonetheless one can still use its $dr^2$ component to
compute the distance to $M_r:=M_{t(r)}$. Thus the integral of $W^{-1/2}$ in this asymptotic
form shows that
\be\lb{eta-dist}
\lim_{r\to\infty}d_g(p_0,M_r)=\infty.
\end{equation}
This completes the proof.
\end{proof}

\subsection{Smooth extension to a singular orbit}
%The phrase ``attaching a bolt" refers to replacing a $4$-manifold with a cohomogeneity one action
%with only regular fibers over an open interval with one admitting a similar action for the same %group
%over a semi-closed interval with a two dimensional singular fiber (the bolt) over the endpoint of %the interval.
For the case at hand, the cohomogeneity one $4$-manifold  with one singular orbit attached can be described as
\[
E(2)\times_{SO(2)}\mathbb{R}^2= (0,\infty)\times E(2)\ \amalg\ \{0\}\times \mathbb{R}^2 ,
\]
where the right $SO(2)$-action is $(g, (T,x))\to (Tg,g^{-1}x)$.
\begin{prop}\label{prop:bolt}
The metric and K\"ahler form corresponding to solutions of \Ref{E2ODEa}-\Ref{E2ODEd} along the unstable curves of the equilibrium points $(q,0,q,0)$, $q>0$, defined on $(-\infty, \eta)$,
can be smoothly extended to $M=E(2)\times_{SO(2)}\mathbb{R}^2$, with the two-dimensional singular orbit $E(2)/SO(2)$ defined over $\xi=-\infty$.
\end{prop}
\begin{proof}
%Consider the manifold $M=E(2)\times_{SO(2)}\mathbb{R}^2$.  This has a left action by $E(2)$ with %regular orbit $E(2)$ and singular orbit $E(2)/SO(2)$.
For any $E(2)$ invariant metric $g$ on $M$, with $r$ the distance along a geodesic perpendicular to the singular orbit,
\[ g = dr^2+g_r. \]
For a metric $g$ of the form \Ref{bianchiAmet}, as usual, let $r=\int_{-\infty}^t a(s)b(s)c(s)\,ds$, then
\[ g = dr^2+a^2\sigma_1^2+b^2\sigma_2^2+c^2\sigma_3^2. \]
The ODE's \Ref{E2ODEa}-\Ref{E2ODEd} in this coordinate become
\begin{align}
\label{rODEa}\frac{da}{dr}&=\frac{a}{2}\left(-\frac{a}{bc}+\frac{c}{ab}\right),\\
\label{rODEb}\frac{db}{dr}&=\frac{1}{2}\left(\frac{a}{c}+\frac{c}{a}\right),\\
\label{rODEc}\frac{dc}{dr}&=\frac{c}{2}\left(\frac{a}{bc}-\frac{c}{ab}+\fr{2\al}{abc}\right),\\
\lb{rODEd}\fr{d\al}{dr}&=\al\fr c{ab}.
\end{align}
From these it is seen that $a$, $b$, $c$ and $\al$ can be extended at $r=0$ so that $a$, $c$
and $\al$ are even and $b$ is odd, as functions of $r$.
Following the notations of Verdiani and Ziller \cite{vz}, the tangent space for $r\neq 0$ splits as
\[ T_p M = \mathbb{R}\partial_r\oplus \mathfrak{k}\oplus\mathfrak{m}, \]
where
\[\mathfrak{k}=\mathrm{span}\{X_2\}, \]
\[ \mathfrak{m}=\mathrm{span}\{X_1,X_3\}=:\ell_1, \]
and we set
\[ V=\mathrm{span}\{\partial_r,X_2\}=:\ell_{-1}'. \]
Now $\exp(\theta X_2)$ acts on both $V$ and $\mathfrak{m}$ as a rotation by $\theta$, so the weights are $a_1=d_1=1$.
The smoothness conditions for $V$ is that $b$ can be extended to an odd function and $b'(0)=1$.
Since we know that $b$ can be extended to be odd, we complete from \Ref{rODEb} the check that
\[ \left.\frac{db}{dr}\right|_{r=0}=\frac{1}{2}\left(\frac{q}{q}+\frac{q}{q}\right)=1.\]
Since $\ell'_{-1}$ and $\ell_1$ are perpendicular, the smoothness conditions in table C of
\cite{vz} are automatically satisfied, while those in
table B there, are
\begin{align}\label{smooth1}a^2+c^2&=\phi_1(r^2),\\
\label{smooth2}a^2-c^2&=r^2\phi_2(r^2),
\end{align}
for some smooth functions $\phi_1$ and $\phi_2$.
Now to see that \Ref{smooth1} is satisfied, note that
\[ a^2+c^2=2ac\frac{db}{dr}.\]
Since $a$, $c$, and $\frac{db}{dr}$ are even, it just remains to check \Ref{smooth2}.
Solving the equations in their power series expansions in $r$ gives $a=q+o(r^4)$, $c=q+o(r^4)$, so that $a^2-c^2$ has the required form.
%We have
%\[ a^2-c^2=c^2\left(\frac{a^2}{c^2}-1\right)\simeq -b^2c^2, \]
%and since $b(0)=0$ and $\frac{db}{dr}|_{r=0}=1$,
Thus $g$ extends to a smooth metric on $M$.

The derivation that the K\"ahler form also extends smoothly proceeds as in \cite{mr},
so we just state the resulting smoothness conditions:

\begin{align*}
cr+ab &= r\phi_2(r^2), \\
cr-ab &= r^3\phi_3(r^2). \\
\end{align*}
The first of these is clear from the oddness/evenness properties of $a$, $b$, $c$.
The above Taylor series expansion of $a$, $c$, in addition to the one for $b$,
namely $b=r+o(r^5)$ easily shows that $cr-ab$ has the required form.
%expand to get $a/c= 1-\frac{b^2}{2}+\mathcal{O}(b^4)$ and $b= r+\mathcal{O}(r^3)$, so
%\[ %cr\left(1-\frac{ab}{cr}\right)=cr\left(1-\frac{b-b^3/2+\mathcal{O}(b^5)}{r}\right)=r^3\phi_3(r^2).\]
%Therefore $\omega$ extends as a smooth form on all of $M$.
\end{proof}

\subsection{Completeness}
\begin{prop}\label{prop:complete}
For the metrics of Proposition~\ref{prop:bolt},
all finite length curves remain inside some compact set.
\end{prop}
The proof here is identical to that in \cite{mr}, and will thus be omitted.
This completes the proof of Theorem \ref{thm4}.

\section{Acknowledgements}

The authors thank Robert~Ream for helpful exchanges pertaining to
the Verdiani-Ziller method.

\appendix

\section{Outline of the derivation of the ODE and PDE systems}
\subsection{Generalized PDEs}
Suppose one is given a $4$-manifold with a frame $\kk$, $\tT$, $\xx$, $\yy$ satisfying
the Lie bracket relations \Ref{brack1}-\Ref{brack3} for functions $A$, $B$, $C$, $D$, $E$, $F$, $G$, $H$, $L$, $N$ on the frame domain. The dual coframe
$\hat\kk$, $\hat\tT$, $\hat\xx$, $\hat\yy$ then satisfies
\begin{align}
d\kf&=-N\xf\wedge\yf-L\kf\wedge\tf,\nonumber\\
d\tf&=-N\xf\wedge\yf-L\kf\wedge\tf,\nonumber\\
d\xf&=-A\kf\wedge\xf-C\kf\wedge\yf-E\tf\wedge\xf-G\tf\wedge\yf,\nonumber\\
d\yf&=-B\kf\wedge\xf-D\kf\wedge\yf-F\tf\wedge\xf-H\tf\wedge\yf.\lb{d-frame}
\end{align}
The vanishing of $d^2$ on the coframe $1$-forms gives four equations, two of which
are identical. Writing, for example, $dN=d_\kk N\kf+d_\tT N\tf+d_\xx N\xf+d_\yy N\yf$ etc.
and separating components yields $12$ scalar equations
\begin{align}
d_\xv L&=0,\qquad d_\yv L=0,\nonumber\\
d_\yv A&=d_\xv C,\qquad d_\yv B=d_\xv D,\qquad d_\yv E=d_\xv G,\qquad d_\yv F=d_\xv H,\nonumber\\
d_\tv N&=NE+NH+LN,\qquad d_\kv N=NA+ND-LN,\lb{nl}\\
d_\tv A&=d_\kv E-AL+CF-EL-GB,\lb{dta} \\
d_\tv B&=d_\kv F-BL+BE+DF-FL-FA-HB,\lb{dtb} \\
d_\tv C&=d_\kv G+AG-CL+CH-EC-GL-GD,\lb{dtc} \\
d_\tv D&=d_\kv H+BG-DL-FC-HL.\lb{dtd}
\end{align}
Adding and subtracting the two equations \Ref{nl}, the two equations \Ref{dta} and \Ref{dtd}
and the two equations \Ref{dtb}-\Ref{dtc}, while using relations \Ref{rels1}-\Ref{rels2},
yields six equations of which only five are independent. The resulting equivalent system is
\begin{align}
&d_\xv L=0,\qquad d_\yv L=0,\lb{ll}\\
&d_\yv A=d_\xv C,\qquad d_\yv B=d_\xv D,\qquad d_\yv E=d_\xv G,\qquad d_\yv F=d_\xv H,\lb{foursome}\\
&d_{\kv+\tv} N=0,\qquad d_{\kv-\tv} N=2N^2-2LN,\lb{nll} \\
&d_\tv (F+G)=-d_\kv (B+C)-(F+G)L+(B+C)L-2(F+G)B+2(B+C)F,\lb{long1}\\
&d_\kv (F+G)=d_\tv (B+C)+(B+C)L+(F+G)L+F^2-G^2+B^2-C^2,\lb{long2}\\
&d_\tv (B-C)=d_\kv (F-G)-(B-C)L-(F-G)L-(B+C)^2-(F+G)^2.\lb{long3}
\end{align}
Assume now that $M$ admits a K\"ahler metric making our frame orthonormal,
which is additionally central. Then, in addition to the above system,
we have equation \Ref{cent}, which we now reproduce:

\begin{align}
&-N(2L+C-H+A-F)[-L(2L+C-H+A-F)\nonumber\\
&+d_{\kk-\tT}L-d_\tT(C-H)+d_\kk(A-F)]\nonumber\\
&-d_\xx(L+C-H)d_\yy(L+A-F)+d_\yy(L+C-H)d_\xx(L+A-F)=\lam.\lb{cent1}
\end{align}
Equations \Ref{ll}-\Ref{cent1} constitute our system in the general case.
With the help of \Ref{ll}, equation \Ref{cent1} can be simplified a little
to the form
\begin{align}
&-N(2L+C-H+A-F)[-L(2L+C-H+A-F)\nonumber\\
&+d_{\kk-\tT}L-d_\tT(C-H)+d_\kk(A-F)]\nonumber\\
&-d_\xx(C-H)d_\yy(A-F)+d_\yy(C-H)d_\xx(A-F)=\lam.\lb{cent2}\end{align}

\subsection{The equations in new variables}
Recall our functions $L$, $N$ along with the four given in \Ref{chan-var}
reproduced here.
\begin{align}
P&=(B-C)+(F-G), &&Q=(B-C)-(F-G),\nonumber \\
R&=\sqrt{(B+C)^2+(F+G)^2}, &&S=\tan^{-1}\left(\frac{B+C}{F+G}\right),\lb{chan-var1}
%-\tan^{-1}k,  \text{ for a constant $k$.}
\end{align}
where $S$ is only defined on the set $\{F+G\}\ne 0$.

In terms of these, we have the inverse transformation
\begin{align}
B&=[(P+Q)+2R\sin S]/4,\qquad C=[-(P+Q)+2R\sin S]/4,\nonumber\\
F&=[(P-Q)+2R\cos S]/4,\qquad G=[-(P-Q)+2R\cos S]/4.\lb{change}
\end{align}

We can write the system \Ref{ll}-\Ref{long3}, \Ref{cent2} in these variables as
follows
\begin{align}
&d_\xv L=0,\qquad d_\yv L=0,\lb{ll1}\\
&d_\yy N+d_\yy(R\cos S)=d_\xx(R\sin S)-d_\xx(P+Q)/2\\
&d_\xx N-d_\xx(R\cos S)=d_\yy(R\sin S)+d_\yy(P+Q)/2\\
&-d_\yy N-d_\yy(R\sin S)=d_\xx(R\cos S)-d_\xx(P-Q)/2\\
&-d_\xx N+d_\xx(R\sin S)=d_\yy(R\cos S)+d_\yy(P-Q)/2\\
&d_{\kv+\tv} N=0,\qquad d_{\kv-\tv} N=2N^2-2LN,\lb{nll1} \\
&d_\tv (R\cos S)=-d_\kv (R\sin S)-RL(\cos S-\sin S)\nonumber\\
&+\frac{1}{2}(P-Q)R\sin S-\frac{1}{2}(P+Q)R\cos S,
\lb{trig11}\\
&d_\kv (R\cos S)=d_\tv (R\sin S)+RL(\sin S+\cos S)\nonumber\\
&+\frac{1}{2}(P-Q)R\cos S+\frac{1}{2}(P+Q)R\sin S,
\lb{trig21}\\
&\frac{1}{2}d_\tv (P+Q) = \frac{1}{2}d_\kv (P-Q)-PL-R^2,\lb{longy11}\\
&-N(2L+N-P/2)[-L(2L+N-P/2)+d_{\kk-\tT}L-(d_\tT N/2-d_\tT(P+Q)/4)\nonumber\\
&+(d_\kk N/2-d_\kk(P-Q)/4)]-(d_\yy N/2-d_\yy(P-Q)/4)(d_\xx N/2-d_\xx(P+Q)/4)\nonumber\\
&+(d_\yy N/2-d_\yy(P+Q)/4)(d_\xx N/2-d_\xx(P-Q)/4)=\lam.\lb{cent3}
\end{align}
The verification is as in \cite{mr}, except that for \Ref{cent3}
we used
\be\lb{intermediate1}
A-F=(N-F+G)/2,\qquad C-H=(N-B+C)/2,
\end{equation}
which follows from \Ref{rels1}-\Ref{rels2}.

Of these equations, \Ref{trig11}-\Ref{trig21} can be simplified
as in \cite{mr} to
\begin{align}\lb{R-S1}
d_\tv R &= -R d_\kv S - RL - \frac{1}{2}(P+Q)R,\qquad
d_\kv R = R d_\tv S + RL + \frac{1}{2}(P-Q)R.
\end{align}
Additionally, \Ref{cent3} can be rewritten as
\begin{align}
&-N(2L+N-P/2)[-L(2L+N-P/2)+d_{\kk-\tT}L+d_{\kk-\tT} N/2-d_{\kk-\tT}P/4+d_{\kk+\tT}Q/4]\nonumber\\
&+d_\yy N d_\xx Q/4-d_\xx N d_\yy Q/4+d_\yy Q\, d_\xx P/8-d_\yy P\, d_\xx Q/8
=\lam.\lb{cent4}
\end{align}

At this point our derivation splits into cases.

\subsection{The case where all functions depend on $\ta$}.
Recall that there exists a local function $\ta$ such that
$\n\ta=\kk-\tT$. Since
\[
d_\xx\ta=0,\qquad d_\yy\ta=0,\qquad d_{\kk+\tT}\ta=0,
\]
it follows that if $A,\ldots H,L, N$ are locally compositions of functions of $\ta$,
the equations \Ref{ll}-\Ref{long3}, \Ref{cent2} simplify to \Ref{nll}-\Ref{long3} without the first
equation in \Ref{nll}, together with
\begin{align}
&-N(2L+C-H+A-F)[-L(2L+C-H+A-F)\nonumber\\
&+d_{\kk-\tT}L-d_\tT(C-H)+d_\kk(A-F)]=\lam.\lb{cent6}
\end{align}
In terms of the variables \Ref{chan-var1} this system takes the form
\begin{align*}
&d_{\kv-\tv} N=2N^2-2LN,\\
&d_{\kk-\tT}R=R(P+2L),\qquad  0=-R(d_{\kk-\tT}S+Q),\\
&d_{\kk-\tT}P=2LP+2R^2,\\
&-N(2L+N-P/2)[-L(2L+N-P/2)+d_{\kk-\tT}(L+N/2-P/4)]=\lam,
\end{align*}
where we have used \Ref{R-S1} as well as $d_{\kk}\ta=1$, $d_{\tT}\ta=-1$.
Alternatively, with a prime denoting differentiation with respect to $\ta$,
since $d_{\kk-\tT}\ta=2$, we can write the system as
\begin{align}
&N'=N^2-LN,\nonumber \\
&R'=R(P/2+L),\qquad  0=-R(2S'+Q),\nonumber\\
&P'=LP+R^2,\nonumber\\
&-N(2L+N-P/2)[-L(2L+N-P/2)+(2L'+N'-P'/2)]=\lam.\lb{cent-onevar}
\end{align}

\subsection{The case $N=0$}

When $N=0$, the only equations that become trivial are \Ref{nll1},
but some equations simplify. We only write the resulting system
in the variables \Ref{chan-var1}. We have
\begin{align}
&d_\xv L=0,\qquad d_\yv L=0,\\
&d_\yy(R\cos S)=d_\xx(R\sin S)-d_\xx(P+Q)/2,\\
-&d_\xx(R\cos S)=d_\yy(R\sin S)+d_\yy(P+Q)/2,\\
-&d_\yy(R\sin S)=d_\xx(R\cos S)-d_\xx(P-Q)/2,\\
&d_\xx(R\sin S)=d_\yy(R\cos S)+d_\yy(P-Q)/2,\\
&d_\yy Q\, d_\xx P/8-d_\yy P\, d_\xx Q/8=\lam,\lb{cent7}\\
&d_{\kk-\tT}R=R(d_{\kk+\tT}S+P+2L),\lb{Skpt}\\
&d_{\kk+\tT}R=-R(d_{\kk-\tT}S+Q),\lb{Skmt}\\
&d_{\kk-\tT}P-d_{\kk+\tT}Q=2LP+2R^2,
\end{align}
where we have written the central curvature equation in \Ref{cent7}.

We see that the system decouples, in the sense that the first six equations
involve only $d_\xx$, $d_\yy$ derivatives, while the last three involve only
$d_{\kk\pm\tT}$. For these last three, recall from \cite{mr} that a rotation
in the planes spanned by $\xx$, $\yy$ allows us to dispense with $d_{\kk\pm\tT}S$
in \Ref{Skpt}-\Ref{Skmt}, simplifying the equations further.

\end{document}